\documentclass[12pt]{amsart}
\usepackage{fontawesome5}

\usepackage{amsmath,amsthm,amssymb}
\usepackage{xcolor}
\usepackage{hyperref}

\hypersetup{
  colorlinks=false,
  linkbordercolor=red,
  pdfborderstyle={/S/U/W 1}
}

\usepackage{graphicx}
\usepackage{color}
\usepackage[noend]{algpseudocode}
\usepackage{tikz}

\theoremstyle{plain}
\newtheorem{thm}{Theorem}[section]
\newtheorem{prop}[thm]{Proposition}
\newtheorem{cor}[thm]{Corollary}

\newtheorem{lem}[thm]{Lemma}
\newtheorem{fact}[thm]{Fact}
\newtheorem{conj}[thm]{Conjecture}

\newcommand{\auto}{\tag{\addtocounter{equation}{1}\theequation\;{\tiny\faCar}}}

\theoremstyle{definition}
\newtheorem{defn}[thm]{Definition}

\theoremstyle{remark}
\newtheorem{rem}[thm]{Remark}

\numberwithin{equation}{section}

\title[Sharp isoperimetric inequalities]{Sharp isoperimetric inequalities on the Hamming cube near the critical exponent}

\author[P. Durcik]{Polona Durcik}
\address{Chapman University, Orange, CA, USA}
\email{durcik@chapman.edu}

\author[P. Ivanisvili]{Paata Ivanisvili}
\address{University of California Irvine, Irvine, CA, USA}
\email{pivanisv@uci.edu}

\author[J. Roos]{Joris Roos}
\address{University of Massachusetts Lowell\\ Lowell, MA, USA}
\email{joris\_roos@uml.edu}

\date{\today}
\begin{document}
\begin{abstract}
An isoperimetric inequality on the Hamming cube for exponents $\beta\ge 0.50057$ is proved, achieving equality on any subcube.  This was previously known for $\beta\ge \log_2(3/2)\approx 0.585$. 
Improved bounds are also obtained at the critical exponent $\beta=0.5$, including a bound that is asymptotically sharp for small subsets.
A key ingredient is a 
new Bellman-type function involving the Gaussian isoperimetric profile which appears to be a good approximation of the true envelope function. 
Verification uses computer-assisted proofs and interval arithmetic.
Applications include progress towards a conjecture of Kahn and Park as well as sharp Poincar\'e inequalities for Boolean-valued functions near $L^1$.
\end{abstract}


\maketitle

\section{Introduction}
Let $n\ge 1$ be an integer.  For a subset $A$ of the Hamming cube $\{0,1\}^n$ and a vertex $x\in A$ let 
$h_A(x)$ be the number of edges connecting $x$ with the complement of $A$ and
let $h_A(x)=0$ if $x\not\in A$. 
One of our main results is the following isoperimetric inequality.

\begin{thm}\label{thm:mainisoperim}
For all $\beta\ge \beta_0=0.50057$ and $A\subset \{0,1\}^n$ with $|A|\le \frac12$, \begin{equation}\label{eqn:isoperimhalf}
\mathbf{E} h_A^\beta \ge |A| (\log_2(1/|A|))^\beta.
\end{equation}
This is an equality if $A$ is a subcube. 
\end{thm}
Here $\mathbf{E}f=2^{-n} \sum_{x\in \{0,1\}^n} f(x)$ and $|A|=\mathbf{E} \mathbf{1}_A$. 
If \eqref{eqn:isoperimhalf} holds for some $\beta$, then it follows for all $\beta'\ge \beta$ (by H\"older's inequality, see Lemma \ref{lem:holder} below). The threshold $\beta_0=0.50057$ is rooted in the failure of a certain variant of Bobkov's two-point inequality for the Gaussian isoperimetric profile 
(see Remark \ref{rem:beta0explanation} and \ref{rem:Jbobkovtwoptfailure} below).

\begin{rem}[Added in revised version, August 2026]
This failure can be overcome to prove \eqref{eqn:isoperimhalf} in the full conjectured range $\beta\ge 0.5$, see \cite{DIRX26}, also see Remark \ref{rem:dirx}.
\end{rem}

\subsection{Motivation, history and further results}
The quantity \[\mathbf{E} h_A^\beta\] should be interpreted as arising from a natural interpolation between vertex boundary measure ($\beta=0$) and edge boundary measure ($\beta=1$) which are classical quantities that are well-understood on the Hamming cube (Harper \cite{Har66}, Bernstein \cite{Ber67}, Hart \cite{Hart76}; a very nice exposition can be found e.g. in Bollob\'as \cite[\S 16]{Bol86}).

Talagrand \cite{Tal93} has initiated the study of $\mathbf{E} h_A^\beta$ for other values of $\beta$ and proved non-trivial dimension-free lower bounds when $\beta=\frac12$. This is a natural and surprising result, since Hamming ball examples show that no such bounds can exist if $\beta<\frac12$ (see e.g. \cite[\S 3]{BIM23}). 
However, Talagrand's bounds are not sharp and while progress has been made (e.g. Bobkov--G\"otze \cite{BG99}, Kahn--Park \cite{KP20}, Beltran--Ivanisvili--Madrid \cite{BIM23}), no sharp lower bound is currently known for the exponent $\beta=1/2$. 

In this paper we are interested in the minimum possible value of the quantity $\mathbf{E} h_A^\beta$ when $|A|$ is fixed and $\beta\in [\frac12, 1]$.
\begin{defn}\label{defn:isoperimprofile}
The \emph{$\beta$-isoperimetric profile} $\mathcal{B}_\beta$ for the Hamming cube is
\begin{equation}
\label{kideverti}
\mathcal{B}_\beta(x) = \inf_{n\ge 1} \inf_{A\subset\{0,1\}^n,\,|A|=x} \mathbf{E} h_A^\beta.
\end{equation}
\end{defn}
\begin{figure}[hb]
\includegraphics[width=12cm]{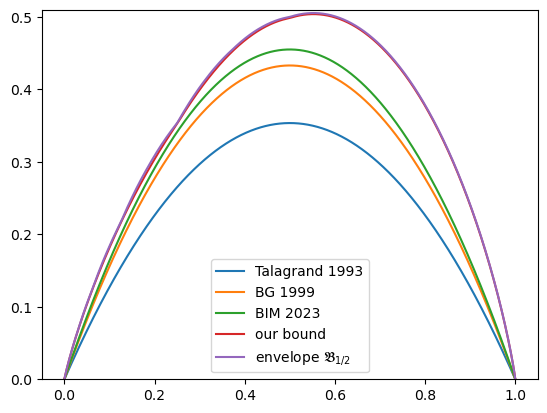}
\caption{Comparison of lower bounds for $\mathcal{B}_{\frac12}$.} \label{fig:boundcompare}
\end{figure}
The definition is restricted to \emph{dyadic rationals} $x$, i.e.
\[ x\in \mathcal{Q}=\{ k 2^{-n}\,:\,n\ge 0, 0\le k\le 2^n \}\subset [0,1],\]
since otherwise there exists no $A$ with $|A|=x$. 
For $\beta=1$ the precise value is known for every $x=k2^{-n}\in\mathcal{Q}$:
\begin{equation}\label{eqn:hartfractal} \mathcal{B}_1(x) = x \Big(n - \frac{2}{k}\sum_{j=1}^{k-1} s(j)\Big),
\end{equation}
where $s(j)$ is the sum of binary digits of $j$ (Hart \cite{Hart76}).
The function $\mathcal{B}_1$ is nowhere differentiable and there is numerical evidence to believe that this is also the case for $\mathcal{B}_\beta$ with $\beta\in (\frac12,1)$ (see \S \ref{sec:computedenvelopes}, in particular Figure \ref{fig:envelopespectrum}). For $\beta=1$ it is not difficult to see that minimizing $\mathbf{E} h_A$ with $|A|$ fixed is equivalent to maximizing the number of edges of $A$ (as an induced subgraph). This no longer holds when $\beta<1$ and no characterization of extremizers is known for any value $\beta<1$. 
The example of a codimension $k$ subcube
\[ A = \{x\in \{0,1\}^n\,:\,x_1=\cdots=x_k=0\} \]
shows that
\[ \mathcal{B}_\beta(2^{-k})\le 2^{-k} k^\beta. \]
Subcubes are extremizers in (\ref{kideverti}) for $\beta=1$, and it is believed that they will remain extremizers for other values of $\beta$ as well. 
\begin{conj}\label{gipoteza}
For all $\beta\ge \frac12$ and all $k\ge 1$,
\[ \mathcal{B}_\beta(2^{-k})= 2^{-k} k^\beta. \]
\end{conj}
Theorem \ref{thm:mainisoperim} proves the conjecture for $\beta\ge 0.50057$.
This was previously shown to hold for $\beta\ge \log_2(\frac32)\approx 0.585$ and $k=1,2$ by Kahn and Park \cite{KP20} and then for $\beta\ge \log_2(\frac32)$ and all $k\ge 1$ by Beltran, Ivanisvili and Madrid \cite{BIM23} (also for $k=1$ when $\beta\ge 0.53$).

\subsection{Results at the critical exponent}
At the critical exponent $\beta=\frac12$ we obtain 
improved estimates for the value $\mathcal{B}_\frac12(\frac12)\le \frac12$.
Talagrand \cite{Tal93} proved that 
\[ \mathcal{B}_{\frac12}(x) \ge C x(1-x) \]
for all $x\in\mathcal{Q}$ with $C=\sqrt{2}$.
This has been improved subsequently by Bobkov--G\"otze \cite{BG99} ($C=\sqrt{3}$) and in \cite[Cor. 1.7]{BIM23} ($C=2\sqrt{2^{\frac32}-2}$). 
The later implies $\mathcal{B}_{\frac12}(\frac12)\ge \frac12 \sqrt{2^{\frac32}-2}\approx 0.455$.
The best possible bound of the form $Cx(1-x)$ is achieved when $C=\frac43 \sqrt{2}$ since for larger values of $C$ the function would contradict the upper bound $\mathcal{B}_{\frac12}(\frac14)\le \frac14 \sqrt{2}$. This would only give $\mathcal{B}_{\frac12}(\frac12)\ge \frac13 \sqrt{2}\approx 0.471$. We break this barrier by a significant margin, coming to within $0.3\%$ of the conjectured value.
\begin{thm}\label{thm:onehalf}
For all $k\ge 1$,
\[ 2^{-k} \sqrt{k}\ge \mathcal{B}_{\frac12}(2^{-k})\ge 0.997\cdot 2^{-k} \sqrt{k}. \]
In particular,
\[ 0.5\ge \mathcal{B}_{\frac12}(\tfrac12)\ge 0.4985. \]
\end{thm}
This is a consequence of a more technical lower bound, see \eqref{eqn:actualisoperimbeta0} below; see Figure \ref{fig:boundcompare} for comparison with previous lower bounds\footnote{Note $\mathfrak{B}_{\beta}$ may be different from $\mathcal{B}_{\beta}$, see Definition \ref{defn:envelope}.}.
While we still do not have any sharp estimates at $\beta=\frac12$ for fixed $x=|A|$, we identify the sharp asymptotic behavior of $\mathcal{B}_{\frac12}$ as $|A|\to 0^+$.

\begin{thm}\label{thm:asymp}
For all $A\subset \{0,1\}^n$, 
\begin{equation}\label{eqn:sharpasymp}
\mathbf{E} \sqrt{h_A} \ge |A| \sqrt{\log_2(1/|A|)+1} - |A|
\end{equation}
and as a consequence,
\[ \mathcal{B}_{\frac12}(x) \sim x\sqrt{\log_2(1/x)} \quad\text{as}\;x\to 0^+.\]
\end{thm}
(Here $f(x)\sim g(x)$ as $x\to x_*$ means $\lim_{x\to x_*} \frac{f(x)}{g(x)}=1$. Note the lower bound is \eqref{eqn:sharpasymp}, while the upper bound follows from Cauchy-Schwarz and the characterization of $\mathcal{B}_1$, which implies $\mathcal{B}_1(x)\sim x \log_2(1/x)$ as $x\to 0+$.)
The bound \eqref{eqn:sharpasymp} is quite bad unless $|A|$ is very close to $0$ (e.g. at $|A|=\frac12$ it proves only $\mathcal{B}_{\frac12}(\frac12)\ge 0.207$). 
However, it features the sharp asymptotic behavior as $|A|\to 0^+$, both in the constant of the leading term and in the power of the logarithm.
This can be compared with the estimate (1.1) of Talagrand \cite{Tal93} which gave a sharp asymptotic in the power of the logarithm, but not the constant. To our knowledge, this is the first sharp estimate proved at the critical exponent $\beta=\frac12$.

We also obtain asymptotics as $|A|\to 1^-$, but we do not know if these are sharp (see Remark \ref{rem:asympnearone}).

\subsection{Classical isoperimetric inequality}
Let us now discuss the case $|A|>\frac12$, which we have omitted in Theorem \ref{thm:mainisoperim}.
The classical isoperimetric inequality states that
\[ \mathbf{E}h_A \ge |A|^* \log_2(1/|A|^*),\]
where $x^*=\min(x,1-x)$, with equality when $A$ is a subcube or complement of a subcube (the inequality follows from \eqref{eqn:hartfractal}). 
The symmetry between $|A|\le \frac12$ and $|A|\ge \frac12$ is natural here because $\mathbf{E} h_A=\mathbf{E} h_{A^c}$. However, generally $\mathbf{E} h_A^\beta\not=\mathbf{E} h_{A^c}^\beta$ for $\beta\not=1$ and complements of subcubes are no longer believed to be extremizers.
Nevertheless, we are able to prove an analogue of the classical isoperimetric inequality
\begin{equation}\label{eqn:classicalisoperimbeta} \mathbf{E}h_A^\beta \ge |A|^* (\log_2(1/|A|^*))^\beta,
\end{equation}
for any given $\beta\in [\beta_0,1]$. This was shown to hold for $\beta=\log_2(3/2)$ in \cite{BIM23}. For $|A|>\frac12$ it turns out that this inequality is far from optimal. 

Let $J$ be the unique smooth function on $(\frac12,1)$ continuous at the endpoints satisfying $J''J=-2, J(\tfrac12)=\tfrac12, J(1)=0$.
The function $J$ plays a central role in this paper and can be expressed in terms of the Gaussian isoperimetric profile $I(x)$, see \eqref{eqn:JintermsofI} below.
\begin{thm}\label{thm:isoperimJ}
For $\beta=\beta_0=0.50057$ and $A\subset \{0,1\}^n$ with $|A|\ge \frac12$,
\begin{equation}\label{eqn:isoperimJ} \mathbf{E} h_A^{\beta} \ge \max(J(|A|), (1-|A|)(\log_2(1/(1-|A|)))^{\beta})
\end{equation}
\end{thm}
In particular, this implies \eqref{eqn:classicalisoperimbeta} for $\beta=\beta_0$. The inequality \eqref{eqn:isoperimJ} improves \eqref{eqn:classicalisoperimbeta} when 
\[\tfrac12<|A|< 1-10^{-10^{380}}.\]
(see Lemma \ref{lem:JvsL_fun}). 

\begin{rem}
The restriction to the hardest case $\beta=\beta_0$ in Theorem \ref{thm:isoperimJ} is only for technical reasons. We can show the conclusion for all $\beta\in [\beta_0,1]$, but omit the details for brevity.
\end{rem}

\subsection{Towards a conjecture of Kahn and Park}

As an immediate consequence of Theorem \ref{thm:mainisoperim} we come closer to a conjecture of Kahn and Park \cite[Conjecture 1.3]{KP20}.
\begin{cor}\label{cor:kahnpark}
Let $(A,B,W)$ be a partition of $\{0,1\}^n$ and assume $|A|=\frac12$. Then
\[ |\nabla(A,B)| + n^{0.50057} |W|\ge \tfrac12, \]
where $|\nabla(A,B)|=2^{-n} \# \{(x,y)\in \mathcal{E}\,:\,x\in A, y\in B\}$ and $|W|=\mathbf{E} \mathbf{1}_W$, and $\mathcal{E}$ is the set of edges of $\{0,1\}^n$.
\end{cor}
Kahn and Park proved this with $\log_3(\frac32)\approx 0.585$ in the exponent of $n$ and in \cite[Cor. 1.6]{BIM23} this was improved to $0.53$, while the conjectured optimal exponent is $0.5$. Corollary \ref{cor:kahnpark} follows from Theorem \ref{thm:mainisoperim} because for $|B\cup W|=|A|=\frac12$, we have
\[ \tfrac12\le \mathbf{E} h_{B\cup W}^{\beta_0} \le \mathbf{E} (h_{B\cup W}\mathbf{1}_B) + \mathbf{E} (h_{B\cup W}^{\beta_0}\mathbf{1}_W) \le |\nabla(A,B)| + n^{0.50057} |W|. \]

\subsection{Sharp Poincar\'e inequalities}
Let $1\leq p\le 2$ and let $C_p$ be the largest constant such that the   $L^p$ Poincar{\'e} inequality 
\[\|\nabla f\|_p \geq C_p \|f-\mathbf{E} f\|_p \]
holds for all functions $f:\{0,1\}^n\to \mathbb{C}$, where $\|f\|_p=(\mathbf{E}|f|^p)^{1/p}$ and 
\[|\nabla f|^2=\sum_{i=1}^n |\tfrac12(f(x)-f(x\oplus e_i))|^2.\]
The value of $C_{p}$  remains unknown except for $p=2$ where $C_2=1$.  In the endpoint case $p=1$, it is known \cite{BELP08}, \cite{ILvHV} that
\[ \sqrt{\tfrac{2}{\pi}}\ge C_1>\tfrac{2}\pi\]
and it is conjectured that $C_1=\sqrt{2/\pi}$, matching Pisier's inequality \cite{pis02} for the Gaussian case. 
The lower bound $C_{1}>\frac{2}{\pi}$ was obtained only recently \cite{ILvHV} with arguments that eventually led to the resolution of Enflo's problem \cite{IVH}. These techniques were later used in \cite{haonan1} and  \cite{Esken1} to obtain sharpening of Poincar\'e inequalities in the quantum setting, and for vector valued functions. 

In \cite{BIM23} it was shown that if we restrict the inequality to Boolean-valued functions $f:\{0,1\}^n\to \{0,1\}$, then 
 the corresponding best constant $C_{B,p}$ 
satisfies $C_{B,1} >C_1$. Specifically, it was shown that $C_{B,1}\ge \sqrt{2^{3/2}-2}>\sqrt{2/\pi}$.
It is conjectured that indicator functions of half-cubes extremize this inequality, which would mean that $C_{B,p}=1$ for all $p\ge 1$.
We obtain an improvement of the best known lower bound for $C_{B,1}$ as well as the sharp bound $C_{B,p} =1$ for $p\geq 2\beta_0$.
\begin{thm}\label{thm:poincare} For all $p\geq 2\beta_0= 1.00114$ and $f:\{0,1\}^n\to \{0,1\}$,
\begin{equation}\label{eqn:poincarealmostL1}
\|\nabla f\|_p\ge \|f-\mathbf{E}f\|_p. 
\end{equation}
Equality is achieved when $f=\mathbf{1}_A$ for a half-cube $A$. 
Moreover, 
\[ \|\nabla f\|_1\ge 0.997 \|f-\mathbf{E}f\|_1. \]
 \end{thm}
This relies on the familiar observation that $\|\nabla \mathbf{1}_A\|_p$ can be written in terms of $\mathbf{E} h_A^{p/2}$ and $\mathbf{E} h_{A^c}^{p/2}$; see \eqref{eqn:twosidedgradient}.
A subtlety is that \eqref{eqn:poincarealmostL1}  does not follow from Theorems \ref{thm:mainisoperim} and \ref{thm:isoperimJ}, but requires the stronger, more technical isoperimetric inequalities \eqref{eqn:actualisoperimhalf} and \eqref{eqn:actualisoperimbeta0} proved below.

\subsection*{Structure of the paper} Here is a brief overview of the different sections and where to find the proofs of the main theorems:
\begin{itemize}
\item[--] In \S \ref{sec:envelope} we reduce the proofs of Theorems \ref{thm:mainisoperim}, \ref{thm:onehalf} and \ref{thm:isoperimJ} to proving certain two-point inequalities for a Bellman-style function, Theorem \ref{thm:maintwopt}. Finding this function was facilitated by numerical approximation of the true envelope function. 
\item[--] In \S \ref{sec:bobkov} we revisit Bobkov's inequality and prove Theorem \ref{thm:asymp}, as well as a variant of Bobkov's two-point inequality used in the proof of Theorem \ref{thm:maintwopt}.
\item[--] In \S \ref{sec:auto} we explain how to automate proving that a function is positive using interval arithmetic. 
To our knowledge it is the first time computer-assisted methods are used in this context.
\item[--] In \S \ref{sec:prelim} we prove various auxiliary estimates used throughout.
\item[--] In \S \ref{sec:maintwopt} we prove the main two-point inequalities, Theorem \ref{thm:maintwopt}. This takes a considerable amount of effort, even with computer-assisted proofs. 
The analysis naturally breaks into several cases; the most critical one is ``Case $J$.I'' in \S \ref{sec:J-I}.
\item[--] In \S \ref{sec:poincare} we derive Theorem \ref{thm:poincare} as a consequence of Theorem \ref{thm:maintwopt}.
\end{itemize}

\subsection*{Acknowledgments} We thank the American Institute of Mathematics (AIM) for funding our SQuaRE workshop from which this project developed. We also thank our fellow SQuaRE members Irina Holmes and Alexander Volberg. J.R. also thanks Rodrigo Ba\~{n}uelos for helpful comments. The authors were  supported in part by grants from the National Science Foundation DMS-2154356 (P.D.), CAREER-DMS-2152401 (P.I.), DMS-2154835 (J.R.). 
We thank an anonymous referee for a thorough reading of the paper and many useful comments.

\section{Envelope functions}\label{sec:envelope}
Let $B:[0,1]\to [0,\infty)$ be given. For $\beta\ge \tfrac12$, $0\le x\le y\le 1$ set
\begin{equation}\label{eqn:Gdef}
G^1_{\beta}[B](x,y) = ((y-x)^{1/\beta}+B(y)^{1/\beta})^\beta + B(x) - 2 B(\tfrac{x+y}2),
\end{equation}
\[ G^2_{\beta}[B](x,y) = y-x + (2^\beta-1) B(y)  + B(x) - 2 B(\tfrac{x+y}2), \]
\[ G_{\beta}[B](x,y) = \max(G^1_\beta[B](x,y), G^2_{\beta}[B](x,y)).  \]

\begin{prop}\label{prop:kahnpark}
Suppose  $B:[0,1]\to [0,\infty)$ satisfies $B(0)=B(1)=0$ and that the following two-point inequality holds:
\begin{equation}\label{eqn:twoptineq}
G_\beta[B](x,y)\ge 0
\end{equation}
for all $0\le x\le y\le 1$. Then $\mathcal{B}_\beta\ge B$.
\end{prop}
This was proved by Kahn and Park \cite{KP20}. The proof is by induction on $n$ (see \cite{KP20}, \cite{BIM23}). It improves on prior induction schemes by Talagrand \cite{Tal93}, Bobkov \cite{Bob97} and Bobkov--G\"otze \cite{BG99}.
The refinement lies mainly in the inclusion of the term $G^2$ which is crucial for our application.
This reduces the proof of Theorems \ref{thm:mainisoperim}, \ref{thm:onehalf} and \ref{thm:isoperimJ} to finding an appropriate function $B$ and verifying the two-point inequality \eqref{eqn:twoptineq}. 
Both steps are difficult obstacles.
Define 
\[ L_\beta(x) = x (\log_2(1/x))^\beta, \]
and let $Q_\beta(x)$ be the unique cubic interpolation polynomial such that $Q_\beta(0)=Q_\beta(1)=0$, $Q_\beta(\frac12)=\frac12$ and $Q_\beta(\frac14)=2^{\beta-2}$. Then
\[ Q_\beta(x) = \tfrac23 x (1-x)(2^{\beta+2}-3+4(3-2^{\beta+1})  x). \]
The polynomial $Q_\frac12$ has been used also in \cite{BIM23}.
Define
\begin{equation}\label{eqn:bbdef}
b_\beta(x) = \left\{ \begin{array}{ll}
    L_\beta(x) & \text{for } x\in [0, \frac14],\\
    Q_\beta(x)
     & \text{for }x\in [\frac14, \frac12],\\
    J(x) & \text{for }x\in [\frac12, 1].
    \end{array}\right.
\end{equation}
The function $J$ is as in Theorem \ref{thm:isoperimJ} and 
can be expressed in terms of the Gaussian isoperimetric profile, see \eqref{eqn:JintermsofI}.

\begin{thm}[Main two-point inequality]\label{thm:maintwopt}
Let $\beta_0=0.50057$ and $c_0=0.997$.
Then for all $0\le x\le y\le 1$:
\begin{equation}\label{eqn:maintwoptbeta0}
G_{\beta_0}[b_{\beta_0}](x,y)\ge 0,
\end{equation}
\begin{equation}\label{eqn:maintwoptonehalf}
G_{\frac12}[c_0 \cdot b_{\frac12}](x,y)\ge 0.
\end{equation}
\end{thm}
\begin{rem}
\eqref{eqn:maintwoptbeta0} continues to hold for all $\beta\in [\beta_0, 1]$ and this can be proved by the same methods, but we do not pursue this here.
\end{rem}
\begin{rem}\label{rem:beta0explanation}
Failure for smaller values of $\beta_0$ occurs at $x=\tfrac12$ and $y>\tfrac12$. 
In particular, the failure only involves the function $J$ (the somewhat suspicious looking cubic $Q_\beta$ surprisingly is not involved in this failure).
$J$ is a very natural candidate for modeling the envelope $\mathfrak{B}_{\frac12}$ (see Definition \ref{defn:envelope}) on $[\frac12,1]$. 
The function $J$ extremizes the two-point inequality locally: letting $y\to x$ in the inequality $G_{\frac12}[B](x,y)\ge 0$ and assuming that $B$ is $C^2$ gives the differential inequality
\[ B'' B \ge -2, \]
so it is natural to consider functions satisfying $B''B=-2$ and numerical differentiation of the approximated envelope $\mathfrak{B}_{\frac12}$  on the interval $[\frac12,1]$ further supports this approach.
The numerical evidence also suggests that $|\mathfrak{B}_{\frac12}(x)-J(x)|\le 3\cdot 10^{-5}$ for $x\in [\frac12,1]$ (see Figure \ref{fig:boundcompare}).
The failure of the inequality for the function $J$ at $\beta=\frac12$ is predicted by an analogue of Bobkov's two-point inequality, see Proposition \ref{prop:twoptbobkov} and Remark \ref{rem:Jbobkovtwoptfailure}. 
\end{rem}

The proof of Theorem \ref{thm:maintwopt} uses computer assistance and is contained in \S \ref{sec:maintwopt}. Theorem \ref{thm:maintwopt} and Proposition \ref{prop:kahnpark} immediately imply:

\begin{cor}\label{cor:actualisoperim}
The following isoperimetric inequalities hold:
\begin{equation}\label{eqn:actualisoperimbeta0} \mathcal{B}_{\beta_0} \ge b_{\beta_0},
\end{equation}
\begin{equation}\label{eqn:actualisoperimhalf}
\mathcal{B}_\frac12 \ge 0.997\cdot b_\frac12. 
\end{equation}
\end{cor}
(Recall Definition \ref{defn:isoperimprofile} and \eqref{eqn:bbdef}.)
\begin{rem}[Added in revised version, August 2026]\label{rem:dirx}
Somewhat counterintuitively, it turns out that by slightly relaxing the requirement $B''B=-2$ to $B''B=-\gamma$ for a $\gamma<2$ one may fix this defect and obtain a Bellman function yielding \eqref{eqn:isoperimhalf} for the conjectured critical exponent $\beta=\frac12$ \cite{DIRX26}. However, the present inequalities \eqref{eqn:actualisoperimbeta0} and \eqref{eqn:actualisoperimhalf} remain somewhat sharper in the regime $|A|>\frac12$ (compare \cite[Remark 2.2]{DIRX26}).
\end{rem}

\subsection{Computed envelopes}\label{sec:computedenvelopes}
In order to find a good candidate $B$ it is helpful to observe that if $B_1, B_2$ satisfy \eqref{eqn:twoptineq}, then so does $\max(B_1,B_2)$. This is a consequence of monotonicity (see Lemma \ref{lem:maxlemma2}) and the same holds for the supremum of a families of such functions $(B_i)_{i\in I}$.
\begin{defn}\label{defn:envelope}
The \emph{$\beta$-envelope} $\mathfrak{B}_\beta$ is the supremum of $B$ over all functions $B$ such that $B(0)=B(1)=0$ and \eqref{eqn:twoptineq} holds for all $0\le x\le y\le 1$ with $x,y\in\mathcal{Q}$.
\end{defn}
By the above observation, the $\beta$-envelope satisfies \eqref{eqn:twoptineq}, so Proposition \ref{prop:kahnpark} implies $\mathcal{B}_\beta\ge \mathfrak{B}_\beta$.
Apart from the case $\beta=1$ it is not known whether the reverse inequality holds.

There is a natural way to numerically approximate the envelope functions $\mathfrak{B}_\beta$.
Figure \ref{fig:envelopespectrum} shows plots of numerically approximated envelopes for different values of $\beta$. Details on this and further explorations will appear elsewhere. These numerical approximations are not used to prove any of our results, but they were helpful in finding the function $b_{\frac12}$ which can be viewed as an approximation of the envelope $\mathfrak{B}_{\frac12}$.

\begin{figure}[ht]
\includegraphics[width=12cm]{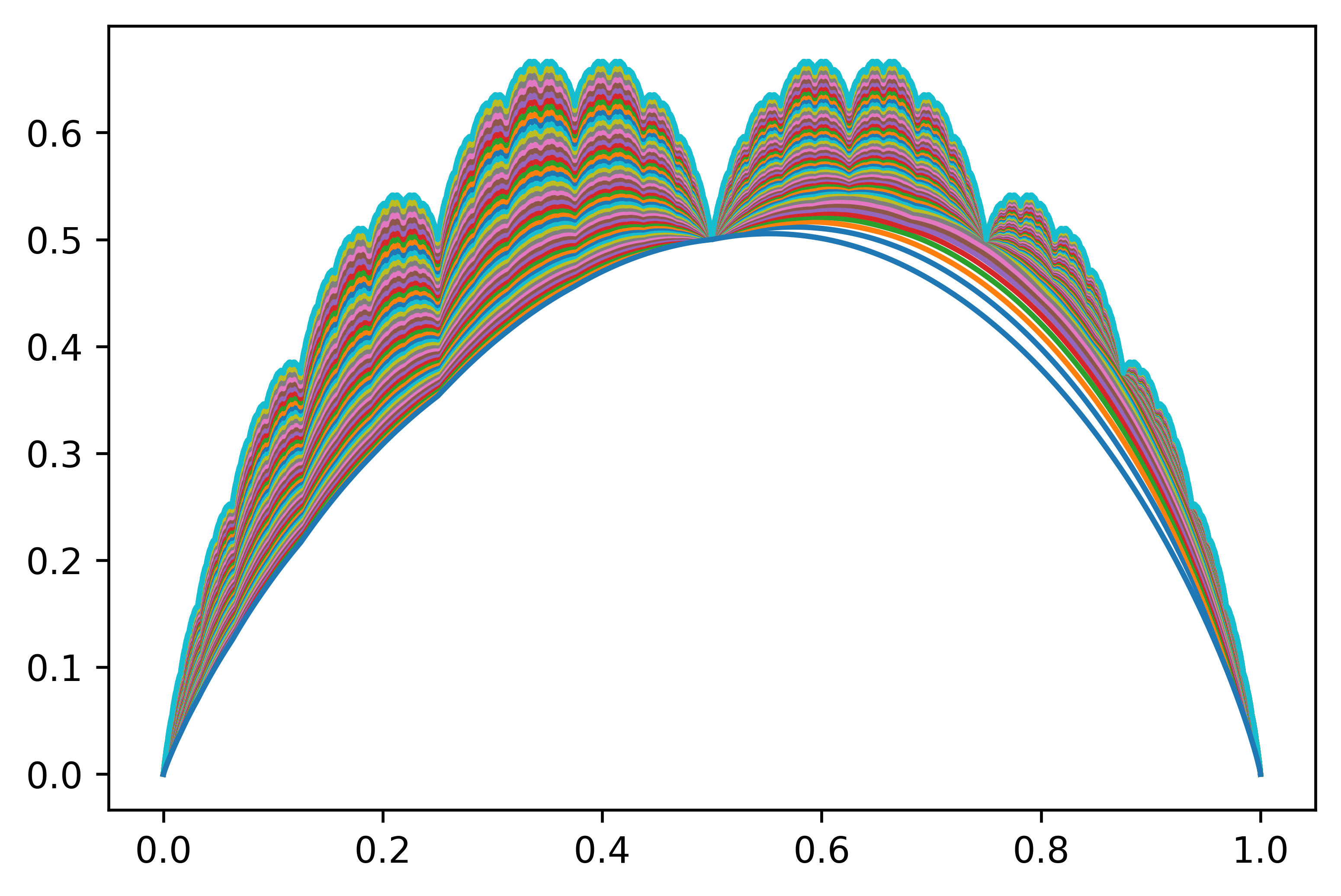}
\caption{Numerically approximated envelopes $\mathfrak{B}_\beta$ for $\beta=\frac12$ (bottom, blue) to $\beta=1$ (top, cyan) with increments in $\beta$ of $0.01$; matching the isoperimetric profile $\mathcal{B}_1$ in \eqref{eqn:hartfractal} when $\beta=1$.}
\label{fig:envelopespectrum}
\end{figure}

\subsection{Proofs of Theorems \ref{thm:mainisoperim}, \ref{thm:onehalf}, \ref{thm:isoperimJ}}\label{sec:mainthmproofs}
To prove Theorem \ref{thm:mainisoperim} it suffices to consider the case $\beta=\beta_0$ (by Lemma \ref{lem:holder}). For $\beta=\beta_0$, \eqref{eqn:isoperimhalf} follows from \eqref{eqn:actualisoperimbeta0} because
\[ b_\beta(x)=Q_\beta(x)\ge L_\beta(x)\quad\text{for}\;x\in [\tfrac14,\tfrac12]. \]
This and various other useful facts about $L_\beta$ and $Q_\beta$ are proved in \S \ref{sec:LandQprop} (see Lemma \ref{lem:LQmax}).
The lower bound in Theorem \ref{thm:onehalf} follows from \eqref{eqn:actualisoperimhalf} while the upper bound follows from letting $A$ be a subcube of codimension $k$.

To prove Theorem \ref{thm:isoperimJ} it suffices to show that 
\begin{equation*}
\tilde{b}_\beta(x) = \left\{ \begin{array}{ll}
    L_\beta(x) & \text{for } x\in [0, \frac14],\\
    Q_\beta(x)
     & \text{for }x\in [\frac14, \frac12],\\
    \max(J(x),L_\beta(1-x)) & \text{for }x\in [\frac12, 1].
    \end{array}\right.
\end{equation*}
also satisfies the two-point inequality
\begin{equation}\label{eqn:btildetwopt}
G_{\beta_0}[\tilde{b}_{\beta_0}](x,y)\ge 0
\end{equation}
for all $0\le x\le y\le 1$. Notice that $J(x)$ and $L_\beta(1-x)$ are equal to $\frac12$ at $x=\frac12$.
Lemma \ref{lem:JvsL_fun} shows that
\begin{equation}\label{eqn:JvsL} J(x)>L_{\beta_0}(1-x)\quad \text{for}\quad x\in (\tfrac12, 1-10^{-10^{380}}).
\end{equation}

The two-point inequality \eqref{eqn:btildetwopt} now follows from Lemma \ref{lem:maxlemma2} by writing $\tilde{b}_{\beta_0}(x)$ as the maximum of $b_{\beta_0}(x)$ and the function 
\[\tilde{L}(x)=L_{\beta_0}(1-x)\mathbf{1}_{x\ge \frac12}.\]
The hypotheses of Lemma \ref{lem:maxlemma2} are satisfied by \eqref{eqn:maintwoptbeta0} and because it is known that the function $\tilde{L}(x)=x\mapsto L_{\beta_0}(1-x)$ satisfies $G_{\beta_0}[\tilde{L}](x,y)\ge 0$ for all $\frac12\le x\le y\le 1$ (see \cite[Lemma 2.3]{BIM23}). The estimate \eqref{eqn:JvsL} means that $b_{\beta_0}(\tfrac{x+y}{2})<\tilde{L}(\tfrac{x+y}{2})$ can only happen when $\frac{x+y}2\ge 1-10^{-10^{380}}$ which in particular implies $y\ge x=2 \tfrac{x+y}{2}-y\ge\frac12$. Thus \eqref{eqn:btildetwopt} is proved and Theorem \ref{thm:isoperimJ} follows.

\section{Variants of Bobkov's inequality}\label{sec:bobkov}

In this section we revisit aspects of Bobkov's classical proof \cite{Bob97} of the Gaussian isoperimetric inequality to achieve two objectives: 
\begin{enumerate}
\item prove a sharp asymptotic estimate, Theorem \ref{thm:asymp}
\item prove a crucial two-point inequality satisfied by the function $J$
\end{enumerate}
This is also informed by Bobkov and G\"otze's work \cite{BG99}.
For functions $f:\{0,1\}^n\to \mathbb{R}$ let
\[  M_i f(x) = (f(x)-f(x\oplus e_i))_+, M f = \sqrt{\sum_{i=1}^n (M_i f)^2}, \]
where $x\oplus y$ denotes bitwise addition, i.e. $x\oplus e_i$ is $x$ with the $i$th bit flipped.
Observe $M \mathbf{1}_A = \sqrt{h_A}$. This notation originates in \cite{Tal93}, \cite{BG99}.
\begin{lem}\label{prop:generaltwopointinduction}
Suppose that $\mathcal{I}\subset [0,\infty)$ is an interval and $B:\mathcal{I}\to \mathbb{R}$. If
\begin{equation}\label{eqn:generalnormtwopoint}
B(\mathbf{E} f) \le \mathbf{E} \sqrt{B(f)^2 + (Mf)^2}
\end{equation}
holds for all $f:\{0,1\}\to \mathcal{I}$, then it holds for all $f:\{0,1\}^n\to \mathcal{I}$ and $n\ge 1$.
\end{lem}
This is a version of Lemma 2.1 of Bobkov--G\"otze \cite{BG99} and follows from it by rescaling. The proof is by induction on $n$, see \cite[Lemma 1]{Bob97}. For $n=1$, \eqref{eqn:generalnormtwopoint} is referred to as a two-point inequality and may be written as
\begin{equation}\label{eqn:Bobkovtwopt}
B\Big(\tfrac{x+y}{2}\Big) \le \tfrac12 \sqrt{(y-x)^2 + B(y)^2} + \tfrac12 B(x),
\end{equation}
where $\max(f(0),f(1))=y\in\mathcal{I}$ and $\min(f(0),f(1))=x\in\mathcal{I}$. Observe that \eqref{eqn:Bobkovtwopt} is equivalent to 
\[G^1_{\frac12}[B](x,y)\ge 0\]
with $G^1_\beta$ as in \eqref{eqn:Gdef}.

\begin{proof}[Proof of Theorem \ref{thm:asymp}]
It is known that the function $L(x)=x\sqrt{\log_2(1/x)}$ satisfies \eqref{eqn:Bobkovtwopt} when $0\le x\le y\le \frac12$ (see \cite[Lemma 2.2]{BIM23}). 
Applying Lemma \ref{prop:generaltwopointinduction} with $\mathcal{I}=[0,\frac12]$ we therefore have
\[ L(\mathbf{E} f)\le \mathbf{E} \sqrt{ L(f)^2 + (Mf)^2} \]
for all $f:\{0,1\}^n\to [0,\frac12]$.
Plugging in $f=\frac12 \mathbf{1}_A$ we have $\mathbf{E} f=\frac12 |A|$ and the inequality becomes
\[ L(\tfrac12|A|) \le \tfrac12 \mathbf{E} (\mathbf{1}_A \sqrt{1 + h_A})\le \tfrac12 (|A| +  \mathbf{E} \sqrt{h_A}) \]
so  the claim follows.
\end{proof}

The two-point inequality \eqref{eqn:Bobkovtwopt} is a variant of Bobkov's  original two-point inequality \cite{Bob97} for the ``two-sided gradient'', i.e., 
\begin{equation}\label{eqn:originalBobkovtwopt}
B\Big(\tfrac{x+y}{2}\Big) \le \tfrac12 \sqrt{\Big(\tfrac{y-x}2\Big)^2 + B(y)^2} + \tfrac12 \sqrt{\Big(\tfrac{y-x}2\Big)^2 + B(x)^2},
\end{equation}
which is closely related and satisfied by the Gaussian isoperimetric profile \cite{Bob97}
\[ I(x) = \varphi(\Phi^{-1}(x)), \]
where $\varphi$ is the standard Gaussian density function and $\Phi$ its cumulative distribution function, i.e.
\[ \varphi(t) = (2\pi)^{-1/2} e^{-t^2/2},\quad \Phi(t) = \int_{-\infty}^t \varphi(s)\,ds.  \]
The function $I$ is also characterized by the conditions 
\[ I(0)=I(1)=0, I\cdot I''=-1. \]
For a real parameter $w>0$ and $x$ with $w^{-1}(1-x)\in [0,1]$, i.e. $x\in [1-w,1]$, we define
\[ J_w(x) = \sqrt{2}\cdot w I(w^{-1}(1-x)). \]
Then $J_w(1)=0$ and $J_w''J_w=-2$. Thus, with $J$ being the function in Theorem \ref{thm:isoperimJ},
\begin{equation}\label{eqn:JintermsofI}
J(x) = J_{w_0}(x),
\end{equation}
where $w_0$ is such that $J_{w_0}(\frac12)=\frac12$. One computes $w_0\in [0.895, 0.896]$.

In his seminal paper \cite{Bob97}, Bobkov showed that $I$ is the  pointwise maximal non-negative continuous function  satisfying \eqref{eqn:originalBobkovtwopt} together with the boundary conditions $I(0)=I(1)=0$.
We are interested in the inequality \eqref{eqn:Bobkovtwopt}.

\begin{lem}\label{lem:Iprimesqconvex}
If a non-negative function $I$ solves $I\cdot I''=-1$, then $I$ is concave and $(I')^2$ is convex.
\end{lem}

\begin{proof}
The first claim follows from $I''=-1/I\le 0$. For the second claim, we compute $((I')^2)' = 2 I' I'' = -2I'/I$ and 
$((I')^2)'' = -2I''/I + 2(I')^2/I^2 = 2 (1+(I')^2)/I^2 \ge 0$.
\end{proof}

\begin{prop}\label{prop:twoptbobkov}
Let $I$ be a non-negative function so that $I''\cdot I=-1$ and $I'\le 0$ hold on an interval $\mathcal{I}$.
Then we have for all $x,y\in \mathcal{I}$ with $x\le y$ that
\begin{equation}\label{eqn:onesidedtwoptI} \sqrt{\tfrac12 (y-x)^2 + I(y)^2} + I(x) \ge 2 I(\tfrac{x+y}{2}).
\end{equation}
\end{prop}

\begin{rem}
The crucial difference to \cite{Bob97} is the additional requirement that $I'\le 0$.
This requirement is necessary: the conclusion  fails if it does not hold.
\end{rem}

\begin{proof}
This is inspired by Bobkov's argument in \cite{Bob97}.
Let $c=\frac{x+y}{2}$ and $u=\frac{y-x}{2}$. Then $x=c-u, y=c+u$.
Let us fix $c\in \mathcal{I}$. Write $\mathcal{I}=[a,b]$.
The variable $u$ ranges in the interval $[0, c_*]$ with 
$c_* = \mathrm{min}(c-a, b-c)$.
We need to show for all $u\in [0, c_*]$:
\[ \sqrt{2u^2 + I(c+u)^2} \ge 2 I(c) - I(c-u)\]
Since $2I(c) - I(c-u)\ge 0$ (by concavity of $I$), this is equivalent to
\[ 2u^2 + I(c+u)^2 \ge 4 I(c)^2 + I(c-u)^2 - 4 I(c) I(c-u). \]
Define 
\[ G_c(u) = I(c+u)^2 - I(c-u)^2 + 2u^2  - 4I(c)^2 + 4 I(c) I(c-u).\]
We have 
\[ G_c'(u) = 2 I'(c+u)I(c+u) + 2 I'(c-u)I(c-u) + 4u - 4 I(c) I'(c-u),\]
\[ G_c''(u) = 2(I'(c+u)^2 - I'(c-u)^2) + 4(1- \tfrac{I(c)}{I(c-u)}) ,\]
where we have used $I\cdot I''=-1$ to obtain the last equation.
Since $I$ is decreasing on $\mathcal{I}$ we have $I(c-u)\ge I(c)$, so
\[ G''_c(u) \ge 2 (I'(c+u)^2 - I'(c-u)^2). \]
Recall that if $f$ is convex, then for all $x\le y$,
$f(y)-f(x) \ge (y-x)f'(x).$
In particular, since $(I')^2$ is convex by Lemma \ref{lem:Iprimesqconvex},
\[ I'(c+u)^2 - I'(c-u)^2 \ge 2u ((I')^2)'(c-u). \]
Observe 
\[ ((I')^2)' = 2 I' \cdot I'' = - 2 I' / I \ge 0, \]
since $I'\le 0$ on the interval $\mathcal{I}$.
Thus we have shown that $G''_c\ge 0$ on $[0, c_*]$.
Since $G'_c(0) = 0$ we must have $G'_c\ge 0$ on $[0, c_*]$.
Since $G_c(0)=0$, also $G_c\ge 0$ on $[0, c_*]$, as required.
\end{proof}

Applying Proposition \ref{prop:twoptbobkov} to  $I=J_w/\sqrt{2}$ we obtain:
\begin{cor}\label{cor:JsatisfiesBobkovtwopt}
Let $0<w\le 1$. 
Then for all $1-\frac{w}{2}\le x\le y\le 1$:
\[ \sqrt{(y-x)^2 + J_w(y)^2} + J_w(x) \ge 2 J_w(\tfrac{x+y}{2}). \]
\end{cor}
\begin{rem}\label{rem:Jbobkovtwoptfailure}
The range of $x,y$ is optimal: if $x<1-\frac{w}{2}$, the inequality fails because $J_w'$ changes sign at $1-\frac{w}2$, so the necessary condition $I'\le 0$ in Proposition \ref{prop:twoptbobkov} is not satisfied.
\end{rem}
Setting $w=w_0$, we obtain in particular that \eqref{eqn:Bobkovtwopt} holds for $B=J$ as long as $y\ge x\ge x_0,$ where
\begin{equation}\label{eqn:x0def}
x_0 = 1-\tfrac{w_0}2 \in [0.552, 0.553].
\end{equation}
We will rely on this in the proof of Theorem \ref{thm:maintwopt} in \S \ref{sec:maintwopt}.

\begin{rem}\label{rem:asympnearone}
Corollary \ref{cor:JsatisfiesBobkovtwopt} and Lemma \ref{prop:generaltwopointinduction} 
imply also that
\[ J(\mathbf{E} f)\le \mathbf{E} \sqrt{ J(f)^2 + (Mf)^2} \]
holds for all $f:\{0,1\}^n\to [x_0, 1]$.
Letting $f=(1-x_0)\mathbf{1}_A + x_0$ we obtain the isoperimetric inequality
\[ \mathbf{E} \sqrt{h_A} \ge (1-x_0)^{-1}(J((1-x_0)|A|+x_0)-\|J\|_\infty (1-|A|)), \]
which is no good when $|A|$ is away from $1$, but improves on \eqref{eqn:actualisoperimhalf} as $|A|\to 1^-$ since the function on the right hand side is asymptotically equivalent to $J(|A|)$.
\end{rem}

\section{Computer-assisted proofs}\label{sec:auto}
It can be quite challenging to prove that a given function of several real variables is strictly positive.
In this section we describe a simple but powerful method to automate positivity proofs for functions on compact rectangles that we will use routinely throughout the paper. This is likely well-known to readers familiar with computer-assisted proofs.

\subsection{Dyadic partitioning}
By a \emph{rectangle} we mean a product of closed intervals
$[a_1, b_1]\times \cdots \times [a_n, b_n]\subset \mathbb{R}^n$ for $a,b\in \mathbb{R}^n$ with $a_i<b_i$ for all $i=1,\dots,n$. We will use the notation $[a,b]$ for such a rectangle.
A map $\underline{f}:[a,b]\times [a,b]\to \mathbb{R}$ is called a {\em tight lower bound} of a function $f:[a,b]\to \mathbb{R}$ if
\[ f(x)\ge \underline{f}(\underline{x},\overline{x})\quad\text{for all}\; x\in [\underline{x}, \overline{x}]\subset [a,b]\]
and $\underline{f}(\underline{x},\overline{x})\to f(x)$ as $|\overline{x}-\underline{x}|\to 0$, $x\in [\underline{x},\overline{x}]$.

A finite partition $\mathcal{P}$ of $[a,b]$ into rectangles such that $\underline{f}(c,d)>0$ for all $[c,d]\in\mathcal{P}$ will be called \emph{admissible}.
The following compactness observation reduces verification of the infinitely many conditions $f(x)>0$ to only finitely many evaluations.
\begin{fact}\label{fact:admissiblepartition}
Suppose a function $f$ on $[a,b]$ has a tight lower bound $\underline{f}$. Then $f>0$ on $[a,b]$ if and only if there exists an admissible partition.
\end{fact}

\begin{proof}
The `if' part follows from the lower bound property and the definition of an admissible partition.
For the `only if' part let $f>0$. Then for every $x\in [a,b]$ exists $\varepsilon_x>0$ such that $\underline{f}(\underline{x},\overline{x})>0$ for all $\underline{x}, \overline{x}\in [a,b]$ with $x\in [\underline{x},\overline{x}]$ and $|\overline{x}-\underline{x}|\le \varepsilon_x$. The collection of $\varepsilon_x/2$-neighborhoods around each point $x$ forms an open cover of $[a,b]$. By the Heine-Borel theorem, there exists a finite subcover, which induces an admissible partition.
\end{proof}
\begin{rem}
It is not required that $f$ is continuous; however, existence of a tight lower bound implies lower semicontinuity of $f$.
\end{rem}

Admissible partitions $\mathcal{P}$ can be practically determined by {\em recursive dyadic partitioning}, where a given rectangle is partitioned into $2^n$ congruent subrectangles called its {\em children} by cutting along the midpoint in each of the $n$ coordinates:

\algnewcommand{\Fail}{\textbf{Fail}}
\begin{algorithmic}
\Procedure{Partition$_n$}{$\underline{f},[a,b]$, $depth$}\Comment $depth$ is initially $0$
\If{$\underline{f}(a,b)>0$} \Return $\{[a,b]\}$
\ElsIf{$depth \ge maxDepth$} \Fail\Comment $maxDepth<\infty$

\Else
\State $\mathcal{P}:=\emptyset$
\For{each child rectangle\; $[a',b']$}:
    \State{Call \Call{Partition$_n$}{$\underline{f}, [a',b'], depth+1$}}
    \If{not failed}: add result to $\mathcal{P}$
    \Else: \Fail
    \EndIf
\EndFor

    \State \Return{$\mathcal{P}$}
\EndIf
\EndProcedure\\

\State \Call{Partition$_n$}{$\underline{f},[a,b]$} := \Call{Partition$_n$}{$\underline{f},[a,b], 0$}\\
\end{algorithmic}
\begin{rem}
We will only use this when $n = 1$ or $n = 2$. In the case of $n = 2$, we will slightly abuse notation and write
$\underline{f}(\underline{x}_1,\overline{x}_1,\underline{x}_2,\overline{x}_2)$ instead of $\underline{f}((\underline{x}_1,\underline{x}_2),(\overline{x}_1,\overline{x}_2))$.
\end{rem}

Running $\textsc{Partition$_n$}(\underline{f},[a,b])$ either fails or returns an admissible partition, thus providing a proof that $f>0$ holds on $[a,b]$. Figure \ref{fig:critpartition} visualizes an example of an admissible partition that is particularly critical to this paper.
A finite value of $maxDepth$ ensures that the procedure terminates.

Moreover, Fact \ref{fact:admissiblepartition} implies that if it is true that $f>0$ holds on $[a,b]$, then the procedure will always succeed in producing a proof of this, provided that $maxDepth$ is chosen large enough (in this paper  $maxDepth=12$ will suffice). In practice, limitations are imposed by memory and time constraints.
Further limitations of the method are the requirement of \emph{strict} positivity as well as the need for a (good enough) tight lower bound $\underline{f}$ given by an explicit formula.

The proof of positivity is established by providing the pair $(\underline{f}, \mathcal{P})$. Verification of admissibility can, in principle, be carried out using any method of rigorous computation, including manual computation. However, the most practical approach is to use a computer. All necessary computations for this paper can be completed within a few seconds on a standard laptop. We use interval arithmetic to account for all numerical and rounding errors inherent in numerical function evaluation using floating-point arithmetic so that the bounds produced by the computer are provably correct and this can be verified by inspection of the source code.

\subsection{Interval arithmetic}
Interval arithmetic is a well-known method for producing rigorous numerical estimates that has gained widespread use in mathematics over recent decades.
A fundamental issue is that generic real numbers are not represented exactly. Instead, one uses \emph{floating point numbers}, which for the purpose of this discussion can be considered to be rationals of the form $k 2^{-n}$ with $n,k\in\mathbb{Z}$, subject to size restrictions. 
The idea of interval arithmetic is to represent (an approximation of) a real number $x$ by a small interval $[\underline{x},\overline{x}]$ with floating point numbers as endpoints so that $x$ is guaranteed to lie inside the interval.
When evaluating a function $f(x)$ we instead evaluate an associated \emph{interval enclosure} $\underline{\overline{f}}(\underline{x},\overline{x})$ which is an interval-valued map with values $[\underline{f}(\underline{x},\overline{x}), \overline{f}(\underline{x},\overline{x})]$ 
so that
\[ f(x)\in \overline{\underline{f}}(\underline{x},\overline{x}) \]
whenever $x\in [\underline{x},\overline{x}]$ and  $\underline{\overline{f}}(x,x)=[f(x),f(x)]$. Arithmetic operations and standard functions are extended to the `degenerate' floating point values $\infty,-\infty,\text{NaN}$ (`not a number', for undefined operations) with the usual semantics. Note that the canonical interval extension of the order relation on real numbers is only a partial order: if $\underline{\overline{x}}>\underline{\overline{y}}$ does not hold for intervals $\underline{\overline{x}},\underline{\overline{y}}$, this does not imply $\underline{\overline{x}}\le \underline{\overline{y}}$. 
Given a formula for $f$ involving only standard arithmetic operations and functions with known piecewise monotonicity, an interval enclosure can be automatically constructed. For further details on interval arithmetic and computer-assisted proofs see e.g. \cite{Tuc11}, \cite{Rump}, \cite{GomezSerrano} and references therein.
\begin{rem}
One could rely on automatically constructed interval enclosures to provide tight lower bounds for use in partitioning. However, in our applications the automatic enclosures are sometimes not quite sufficient. Because of this and for clarity, we prefer to always manually provide explicit tight lower bounds.
\end{rem}

\subsection{Source code}
Each use of $\textsc{Partition}_n$ throughout the paper has been implemented using \emph{FLINT/Arb}, an open source library for arbitrary precision interval arithmetic that produces provably correct error bounds \cite{Arb}, \cite{Flint}.
The code verifying numerical claims for this paper is available on GitHub at
\[
\texttt{\href{https://github.com/roos-j/dir24-isoperim}{https://github.com/roos-j/dir24-isoperim}}
\]
Every claim verified in this manner will be tagged with ({\tiny\faCar}).
For readers who prefer different methods of verification, we provide admissible partitions for each of these claims as an ancillary file to our arXiv submission (\texttt{partitions.py}; this file is automatically generated by the verification code).\\

\section{Auxiliary estimates}\label{sec:prelim}
In this section we collect various estimates that are used in this paper and may also be useful in the future.

\subsection{Small exponents versus large exponents}

\begin{lem}\label{lem:holder}
If $\mathbf{E} h_A^\beta\ge B(|A|)>0$ holds for some $\beta>0$, then for all $\tilde{\beta}\ge \beta$,
\[ \mathbf{E} h_A^{\tilde{\beta}}\ge B(|A|)^{\frac{\tilde{\beta}}\beta} |A|^{-\frac{\tilde{\beta}}{\beta}+1}. \]
In particular, if \eqref{eqn:isoperimhalf} holds for $\beta_0$, then it holds for all $\beta\ge \beta_0$.
\end{lem}
\begin{proof}
Let $p=\tilde{\beta}/\beta\ge 1$. By the assumption and H\"older's inequality,
\[ B(|A|)\le \mathbf{E} h_A^{\beta} \le (\mathbf{E} h_A^{p\beta})^{\frac1p} |A|^{1-\frac1p}, \]
which implies the claim.
\end{proof}

\subsection{Two-point inequalities}

Suppose we are trying to prove an inequality of the form
\begin{equation}\label{eqn:maxtwofctineq}
\mathcal{L}(x,y,B(x),B(y))\ge B(\tfrac{x+y}{2})
\end{equation} 
for some function $B$ on $[a,b]$ with $\mathcal{L}(x,y,u,v)$ monotone increasing in $u,v$ for all $a\le x\le y\le b$.
A well-known observation is that if $B_1,B_2$ satisfy \eqref{eqn:maxtwofctineq} for all $a\le x\le y\le b$, then the same holds for $B=\max(B_1,B_2)$. We will use this observation with minimal hypotheses as follows.

\begin{lem}\label{lem:maxlemma2}
Assume for $a\le x\le y\le b$ that
\begin{enumerate}
\item \eqref{eqn:maxtwofctineq} holds for $B_1$ when $B_1(\frac{x+y}{2})\ge B_2(\frac{x+y}2)$
\item \eqref{eqn:maxtwofctineq} holds for $B_2$ when $B_1(\frac{x+y}{2})< B_2(\frac{x+y}2)$
\end{enumerate}
Then \eqref{eqn:maxtwofctineq} holds for $B=\max(B_1,B_2)$ for all $a\le x\le y\le b$.
\end{lem}

\begin{proof}
Fix $a\le x\le y\le b$. If $B_1(\frac{x+y}{2})\ge B_2(\frac{x+y}2)$, then
\[ \mathcal{L}(x,y,B(x),B(y))\ge \mathcal{L}(x,y,B_1(x),B_1(y)) \ge B_1(\tfrac{x+y}{2}) = B(\tfrac{x+y}{2}) \]
and the case $B_1(\frac{x+y}{2})< B_2(\frac{x+y}2)$ is analogous.
\end{proof}

\begin{lem}\label{lem:GLvsGR}
Let $\beta\in [\frac12,1)$ 
and $t,s>0$.  Then
\[ (t^{1/\beta}+s^{1/\beta})^\beta \ge t + (2^\beta-1) s\] 
holds if and only if $t\le s$.
\end{lem}

\begin{proof}
Dividing by $s>0$, the inequality can be equivalently written as 
\[ ( ( ts^{-1})^{\tfrac{1}{\beta}}+1 )^\beta  \ge ts^{-1} + 2^\beta-1. \] 
Writing $a=ts^{-1}$, raising the inequality to the power $\frac{1}{\beta}$, and subtracting the right-hand side, it suffices to show that
\[  a^{\frac{1}{\beta}}+1  - (a + 2^\beta-1)^{\frac{1}{\beta}}  \ge 0 \] 
holds 
if and only if $0\leq a\leq 1$. 
When $a=1$, the left-hand side vanishes. Thus, it suffices to show that the left-hand side is strictly decreasing in $a>0$. 
Differentiating the left-hand side in $a$, it suffices to show    
\[a^{\frac{1}{\beta}-1}  - (a + 2^\beta-1)^{\frac{1}{\beta}-1} < 0\]
for each  $a>0$, which  holds because $\tfrac{1}{\beta}-1> 0$. 
\end{proof}
Recalling the definition of $G_\beta$ in \eqref{eqn:Gdef}, Lemma \ref{lem:GLvsGR} implies
\[ G_\beta[B](x,y) = \left\{\begin{array}{ll}
G^1_\beta[B](x,y), & \text{if}\;y-x\le B(y),\\
G^2_\beta[B](x,y), & \text{if}\;y-x\ge B(y).
\end{array}\right. \]
In particular, $G_\beta[B](x,y)=G^1_\beta[B](x,y)$ for $(x,y)$ near the diagonal $\{x=y\}$.
On the diagonal we have $G^1_\beta[B](x,x)=0$, so near the diagonal it is natural to consider the variable $h=y-x$.
The following lower bound will be 
sufficient to bound all near diagonal cases.
\begin{lem}[Near diagonal lower bound]\label{lem:neardiaglowerbd}
Let $B$ be a nonnegative function defined on an interval $\mathcal{I}$.
Let $\beta\in (0,1)$. For $h\ge 0$, and $x,x+h\in\mathcal{I}$,
\[ G^1_\beta[B](x,x+h) = \beta B(x+h)^{1-\frac1\beta} h^{\frac1\beta} + [B(x)+B(x+h)-2B(x+\tfrac12h)] - E_1, \]
where $E_1$ satisfies
\[ 0\le E_1 \le \tfrac12 \beta (1-\beta) B(x+h)^{1-\frac2\beta} h^{\frac2\beta} = O(h^{\frac2\beta}). \]
\end{lem}

\begin{rem}
In our applications, $B$ is always concave, so 
\[B(x)+B(x+h)-2B(x+h/2)\le 0.\]
\end{rem}

\begin{proof}
It follows from Taylor's theorem that for $a\ge 0$,
\[ (1+a)^\beta = 1 + \beta a - \tfrac12 \beta(1-\beta) \widetilde{a}^2 \]
for some $\widetilde{a}\in [0,a]$.
Letting $a=(h B(x+h)^{-1})^{\frac1\beta}$,
\[ (h^{\frac1\beta} + B(x+h)^{\frac1\beta})^\beta = B(x+h) (1+a)^\beta = B(x+h) + \beta B(x+h)^{1-\frac1\beta} h^{\frac1\beta} - E_1, \]
with
\[ E_1 = \tfrac12 B(x+h) \beta(1-\beta) \widetilde{a}^2\ge 0. \]
$E_1$ satisfies the claimed bound, because
$\widetilde{a}^2 \le a^2=h^{\frac2\beta} B(x+h)^{-\frac2\beta}$.
\end{proof}

\subsection{\texorpdfstring{Properties of $L_\beta$ and $Q_\beta$}{Properties of L and Q}}\label{sec:LandQprop}
Let us record the first few derivatives of $L_\beta(x)=x (\log_2(1/x))^\beta$:
\begin{equation}\label{eqn:DL}
L'_\beta(x) = (\log 2)^{-\beta} (\log \tfrac1x)^{-1+\beta} (-\beta +\log \tfrac1x),
\end{equation}
\begin{equation}\label{eqn:DDL}
L''_\beta(x) = -\beta (\log 2)^{-\beta} x^{-1} (\log \tfrac1x)^{-2+\beta}(1-\beta+\log(x^{-1})),
\end{equation}
\begin{equation}\label{eqn:DDDL}
L'''_\beta(x) = \beta (\log 2)^{-\beta} x^{-2} (\log\tfrac1x)^{-3+\beta} (\beta(3-\beta)-2 + (\log\tfrac1x)^2).
\end{equation}
Note $\log$ means logarithm with the natural base $e$.
These will be used to verify the following properties of $L_\beta$.
\begin{lem}\label{lem:Lprop}
Let $\beta\in [\frac12,1]$.
Then
\begin{enumerate}
\item the function $x\mapsto L_\beta(x)$ is strictly concave on $[0,1]$.
\item if $x\in (0,e^{-1})$, then $L''_\beta(x)$ is decreasing in $\beta$.
\item if $x\in (0,e^{-\sqrt{3}/2})$, then $L'''_\beta(x)$ is positive and increasing in $\beta$.
\item if $x\in (0,\tfrac{1}{2}]$,  $L^{(4)}_\beta(x)$ is negative. 
\end{enumerate}
\end{lem}
\begin{rem}
Note $e^{-\sqrt{3}/2}>e^{-1}>\frac14$. \end{rem}
\begin{proof}
(1) This follows from \eqref{eqn:DDL}.

(2) We show that
\[ \partial_\beta \log(-L''_\beta(x)) > 0 \]
which means that $\beta\mapsto \log(-L''_\beta(x))$ is increasing in $\beta$, so $L''_\beta(x)$ is decreasing in $\beta$.
The left hand side equals
\[ \beta^{-1}(1-\beta+\log\tfrac1x)^{-1} \ell(x,\beta), \]
where 
\[ \ell(x,\beta)= -a_2(x)\beta^2 + a_1(x)\beta + a_0(x) \]
and
\begin{align*}
a_0(x) = & 1+\log\tfrac1x>0,\\
a_1(x) = &  -2 +
    \log\tfrac1{\log 2} + (\log\tfrac1x) (\log\tfrac1{\log 2}) + \log\log\tfrac1x\\
    &
    + 
    (\log\tfrac1x) (\log\log\tfrac1x) > a_1(e^{-1}) ,\\
a_2(x) =& \log\tfrac1{\log 2}+\log\log\tfrac1x>0.
\end{align*}
Notice that $\partial_\beta^2 \ell(x,\beta)=-2a_2(x)<0$, so $\beta\mapsto \ell(x,\beta)$ is concave and thus
\[ \ell(x,\beta)\ge \min(\ell(x,\tfrac12),\ell(x,1)). \]
Observe from definition that $x\mapsto \ell(x,\tfrac12)$ and $x\mapsto \ell(x,1)$ are decreasing 
functions, so
\[ \ell(x,\beta)\ge \min(\ell(e^{-1},\tfrac12),\ell(e^{-1},1)). \]
Now it only remains to evaluate
\[ \ell(e^{-1},\tfrac12) = 1+\tfrac34 \log \tfrac1{\log 2} > 0,\]
\[\ell(e^{-1},1)=\log\tfrac{1}{\log 2} > 0. \]
(3) By \eqref{eqn:DDDL} it suffices to show
\[ \beta(3-\beta) - 2 + (\log\tfrac1x)^2 > 0. \]
Observe that the function on the left hand side is  positive for small enough $x$ and strictly decreases in $x$.
Also, the function $\beta\mapsto \beta(3-\beta)$ is increasing since its derivative is $3-2\beta\ge 1$.
Solving
\[ \beta(3-\beta) -2 + (\log\tfrac1x)^2 = 0 \]
with $\beta=\frac12$ for $x$ gives $x=e^{-\sqrt{3}/2}$.
(This also shows that $L'''_\beta(x)$ is increasing in $\beta$.)\\

(4) From \eqref{eqn:DDDL} we compute
 \[ L_\beta^{(4)}(x) = x^{-3}\beta  (\log 2)^{-\beta } (\log \tfrac{1}{x})^{\beta -4}  f_\beta(x), \]
 where
 \[ f_\beta(x) = -(3-\beta) (2-\beta) (1-\beta) +  (1-\beta -2 \log \tfrac{1}{x})(\log \tfrac{1}{x})^2 +2 (2-\beta) (1-\beta)\log \tfrac{1}{x}, \]
so it suffices to show that  
$f_\beta(x)$ is negative. 

To see this, we note that 
\[\partial_\beta f_\beta(\tfrac{1}{2}) = 3 \beta^2+4  (\log 2 -3)\beta +11-(\log 2)^2-6 \log 2 \]
which has both roots strictly greater than $1$ and is thus positive for $\beta\in [\tfrac{1}{2},1]$, so $\beta\mapsto f_\beta(\frac12)$ is increasing.  
Evaluating
\begin{equation*}
    f_{1} (\tfrac12) <-0.6,  
\end{equation*}
this implies $f_\beta(\tfrac{1}{2})<0$ for all $\beta\in [\tfrac{1}{2},1]$.
Thus it suffices to show that 
\[f_\beta'(x)=2x^{-1} \left(3 (\log \tfrac{1}{x})^2 -(1-\beta) \log \tfrac{1}{x}  -\beta ^2+3 \beta -2 \right)\]
is non-negative on $(0,\tfrac{1}{2}]$. For this it suffices to show 
\[3 (\log \tfrac{1}{x})^2  -(1-\beta) \log \tfrac{1}{x} - \beta ^2+3 \beta -2\geq 0\]
 on $(0,\tfrac{1}{2}]$.
 The left-hand side is increasing in $\beta$, so it suffices to verify the inequality for $\beta=\tfrac{1}{2}$. Setting also $u=\log(\tfrac{1}{x})$, it   suffices to show
\[12 u^2  -2 u - 3\geq 0 \]
for $u\in [\log2,\infty)$. This can be verified  by solving for $u$ and noting that the solutions are not greater than $\log 2$. 
\end{proof}

We turn our attention to the cubic polynomials $Q_\beta$ which are defined by
\[ Q_\beta(x) = \tfrac23 x (1-x)(2^{\beta+2}-3+4(3-2^{\beta+1})  x) \]
so that $Q_\beta(0)=Q_\beta(1)=0, Q_\beta(\frac12)=\frac12$ and $Q_\beta(\frac14)=2^{\beta-2}$.
For $\beta~\in~[\tfrac{1}{2},1]$ we denote
\begin{equation}\label{eqn:alpha01def}
\alpha_0 = 2^{2 + \beta} - 5 \ge 2^{2.5}-5>0,
\end{equation}
\[ \alpha_1 = 3-2^{1 + \beta}.  \]  
Note that $\alpha_1>0$ if and only if $\beta<\log_2(\frac32)$. 
We record some derivatives of $Q_\beta$ for future use:
\begin{equation}\label{eqn:DQ}
Q_\beta'(x) =\tfrac{2}{3}(2^{2+\beta}-3) - 4\alpha_0x-8\alpha_1 x^2,
\end{equation}
\begin{equation}\label{eqn:DDQ}
Q''_\beta(x)= -4\alpha_0 - 16\alpha_1 x,
\end{equation}
\begin{equation}\label{eqn:DDDQ}
Q'''_\beta(x)=-16\alpha_1,
\end{equation}
\begin{equation}\label{eqn:Qbetaderiv}
\partial_\beta Q_\beta(x) = \tfrac13 \log(2) 2^{3 + \beta} x (1 - x)  (1 - 2 x).
\end{equation}

\begin{lem}\label{lem:Qbprop}
Let $x\in [0,\tfrac12]$ and  $\beta\in[\frac12,\log_2(\frac32))$. Then:
\begin{enumerate}
\item $Q_\beta(x)\ge 0$ and $Q_\beta$ is increasing in $\beta$.
\item $Q'_\beta(x)> 0$ and if $x\in[\frac14,\frac12]$, then $Q'_\beta$ is decreasing in $\beta$
\item $Q''_\beta(x)<0$, $Q'''_\beta(x)<0$ and $Q''_\beta(x)$ is decreasing in $\beta$
\end{enumerate}
\end{lem}

\begin{proof}
    (1) This follows from \eqref{eqn:Qbetaderiv}.
    
    (2) From \eqref{eqn:DQ} one can write
    \[Q'_\beta(x)= \tfrac{2}{3}(1-2x)(2^{\beta+2}-3+4\alpha_1 x) + \tfrac{8}{3}\alpha_1 x(1-x)>0 \]
for $x\in [0,\frac12]$ and $\beta\in[\frac12,\log_2(\frac32)]$.
    Also,
    \[\partial_\beta Q'_\beta(x) = \tfrac{1}{3}(\log 2)2^{3+\beta}(6x^2-6x+1). \]
This is negative if $x\in [\tfrac{1}{4},\tfrac{1}{2}]$, because
\[ 6x^2-6x+1=(x-\tfrac{1}{2}-\tfrac{\sqrt{3}}{6}) (x-\tfrac{1}{2}+\tfrac{\sqrt{3}}{6}).\]

    (3)  The first part follows from \eqref{eqn:DDQ}, \eqref{eqn:DDDQ}. To see that $Q_\beta''(x)$ is decreasing in $\beta$, compute    \[\partial_\beta Q''_\beta(x) = -\log(2)2^{4+\beta}(1-2x)\le 0, \]
    since $x\leq \tfrac{1}{2}$. 
\end{proof}

\begin{lem}\label{lem:LQmax}
For $\beta\in [\tfrac{1}{2},\log_2\frac32]$ we have
\begin{enumerate}
    \item $L_\beta(x) \ge Q_\beta(x)$ for $x\in [0,\tfrac14],$
    \item $L_\beta(x) \le Q_\beta(x)$ for  $x\in [\tfrac 14,\tfrac12].$
\end{enumerate}
\end{lem}
\begin{rem}
The conclusions in the lemma continue to hold for all $\beta\in [\frac12,1]$.
\end{rem}
\begin{proof}
    We make use of Lemma \ref{lem:Lprop} and  \ref{lem:Qbprop}.

    (1)  We have $L_\beta'''-Q_\beta'''> 0$ on $[0,\frac14]$,  
   so  $L_\beta'-Q_\beta'$ is strictly convex there. Also $Q_\beta'(0)>0$ and $L'_\beta(x)\to\infty$ as $x\to 0^+$. Moreover, 
    \[ L_\beta'(\tfrac{1}{4})-Q_\beta'(\tfrac{1}{4}) = 2^{\beta } ( \tfrac{4}{3}-\tfrac{\beta}{\log 4} )-\tfrac{3}{2} <0,\]
    where the last inequality  follows from evaluating
    \[L_{\frac{1}{2}}'(\tfrac{1}{4})-Q_{\frac{1}{2}}'(\tfrac{1}{4})<-0.12\]
   and the fact that   $L_\beta'-Q_\beta'$ is decreasing in $\beta$, which in turn follows from 
    \[\partial_\beta(L_\beta'-Q_\beta')(\tfrac{1}{4}) = \tfrac{1}{\log 64}2^\beta (-3 - \beta \log 8 + 8(\log 2)^2)\]
    and \[-3 - \beta \log 8 + 8(\log 2)^2 \leq -3 - \tfrac{1}{2} \log 8 + 8(\log 2)^2 <-0.06.\]
   Therefore, $L_\beta'-Q_\beta'$ has a unique zero $\tilde{x}$ on $[0,\tfrac{1}{4}]$. Thus, $L_\beta-Q_\beta$ has a unique local extreme on this interval, which is a maximum  since $L_\beta'-Q_\beta'$ is increasing on $[0,\tilde{x}]$ and decreasing on $[\tilde{x},\tfrac{1}{4}]$. Together with $L_\beta(0)-Q_\beta(0)=L_\beta(\tfrac{1}{4})-Q_\beta(\tfrac{1}{4})=0$ this shows $L_\beta-Q_\beta\geq 0$ on $[0,\tfrac{1}{4}]$. \\

 (2) 
 $L_\beta^{(4)}\leq 0$. Thus, $L_\beta''-Q_\beta''$ is concave. We compute
 \[L_\beta''(\tfrac{1}{2})-Q_\beta''(\tfrac{1}{2})=(\log 2)^{-2}(2 \beta (\beta-   1-\log 2)+4 (\log 2)^2).\]
 This is positive which can be seen  by computing $\partial_\beta(L_\beta''(\tfrac{1}{2})-Q_\beta''(\tfrac{1}{2}))=2(\log 2)^{-2}(2\beta-1-\log 2)<0$ 
 and 
 \[L_{\log_2(3/2)}''(\tfrac{1}{2})-Q_{\log_2(3/2)}''(\tfrac{1}{2})>1.3.\]
Moreover, 
 \[
  L_\beta''(\tfrac{1}{4})-Q_\beta''(\tfrac{1}{4})
  = 8 \left(2^{\beta }-1\right)+(\log 2)^{-2}\bigl(2^{\beta } \beta  (\beta -1-\log 4)\bigr)
 \]
 is increasing in $\beta$. Indeed, we have 
 \[
  \partial_\beta\bigl(L_\beta''(\tfrac{1}{4})-Q_\beta''(\tfrac{1}{4})\bigr) =
 \]
 \[
  = \tfrac{2^{\beta }}{(\log 2)^2} \left(\beta ^2 \log 2+2 \beta -\beta  \log 2 (1+\log 4)-1-\log 4+8 (\log 2)^3\right)
 \]
 and one can see that $\beta ^2 \log 2+2 \beta -\beta  \log 2 (1+\log 4)$ is increasing in $\beta$,  positive at $\beta=\tfrac{1}{2}$, and  $8 (\log 2)^3-1-\log 4>0$. 
Moreover, 
 \begin{equation*}
    L_{\frac{1}{2}}''(\tfrac{1}{4})-Q_{\frac{1}{2}}''(\tfrac{1}{4})>0.5.
\end{equation*}
Therefore,
  $L_\beta''-Q_\beta''\geq 0$, so $L_\beta-Q_\beta$ is convex. Since $L_\beta(\tfrac{1}{4})-Q_\beta(\tfrac{1}{4}) = L_\beta(\tfrac{1}{2})-Q_\beta(\tfrac{1}{2})=0$, we have $L_\beta-Q_\beta\leq 0$ as desired.
\end{proof}

\subsection{Lower bounds for the Gaussian isoperimetric profile}
Recall that $I=\varphi\cdot \Phi^{-1}$ denotes the Gaussian isoperimetric profile.
It is well-known that as $x\to 0^+$,
\[ I(x) \sim \sqrt{2}\cdot x \sqrt{\log(1/x)} \]
(Here $f(x)\sim g(x)$ means $\lim_{x\to 0^+} \frac{f(x)}{g(x)}=1$.)
We shall need a quantitative lower bound.

\begin{prop}\label{prop:Ilowerbd}
For all $0<x\le \frac15$:
\[I(x)\ge \sqrt{2}\cdot x\sqrt{\log(1/x)}\Big(1- \tfrac12 \tfrac{\log\log(1/x)}{\log(1/x)}-\tfrac{\log(2\pi^{1/2})}{\log(1/x)}\Big).\]
\end{prop}

This should be well-known, but since we could not find a reference, we provide a proof.
The estimate $I(x)\le \sqrt{2}\cdot x\sqrt{\log(1/x)}$ also holds and can be proved along the same lines, but we do not need this here.

\begin{proof}[Proof of Proposition \ref{prop:Ilowerbd}]
Let $t<0$. Integration by parts shows
\[ \int_{-\infty}^t e^{-s^2/2}\,ds 
= -t^{-1} e^{-t^2/2} - \int_{-\infty}^t s^{-2} e^{-s^2/2}\,ds. \]
Thus,
\[ \Phi(t) = |t|^{-1} \varphi(t) - \int_{-\infty}^t s^{-2} \varphi(s)\,ds \le |t|^{-1} \varphi(t). \]
Applying $I=\varphi\circ\Phi^{-1}$ (an increasing function on $[0,\tfrac12]$) on both sides,
\[ \varphi(t) \le I(|t|^{-1} \varphi(t)) \]
at least as long as $x=|t|^{-1}\varphi(t)\le 1/2$. 
This means
\[ I(x)\ge x\cdot |t|. \]
Rewriting
\[ x|t| = (2\pi)^{-1/2} e^{-t^2/2}, \]
we have
\[ |t| = \sqrt{2} \sqrt{\log(1/x)-
\log((2\pi)^{1/2}|t|)}  \]
(the expression under the square root is positive because it equals $t^2/2$). The assumption $x\le \frac12$ implies $|t|>0.64$ which also means $(2\pi)^{1/2}|t|>1$, and thus
\[ |t|\le \sqrt{2}\sqrt{\log(1/x)}. \]
This, in turn, implies
\[\log((2\pi)^{1/2}|t|)\le \log(2 \pi^{1/2} \sqrt{\log(1/x)}) = \log(2\pi^{1/2}) + \tfrac12 \log\log(1/x),  \]
and therefore we can write
\[ \sqrt{2} \sqrt{\log(1/x)} \sqrt{1-\varepsilon} = |t|,\]
where $\varepsilon = (\log(1/x))^{-1} \log((2\pi)^{1/2}|t|)$
satisfies
\[ \varepsilon \le \tfrac{\log(2\pi^{1/2})}{\log(1/x)} + \tfrac12 \tfrac{\log\log(1/x)}{\log(1/x)}\to 0\quad\text{as}\quad x\to 0^+. \]
In particular, $\varepsilon<1$ if $x\le \frac15$.
Thus,
\[ I(x)\ge \sqrt{2}\cdot x\sqrt{\log(1/x)} \sqrt{1-\varepsilon}, \]
and using $\sqrt{1-\varepsilon} \ge 1 - \varepsilon$ (since $\varepsilon\le 1$) the claim follows.
\end{proof}

We record the following consequence that will be needed in \S \ref{sec:poincare}.
\begin{cor}\label{prop:Jlowerbd}
For $0<x\le \tfrac1{64}$,
\[ J(1-x)\ge x\,\sqrt{\log(w_0/x)}, \]
where we recall that $w_0\in (0,1)$ is the unique value such that $J_{w_0}(\frac12)=\frac12$.
\end{cor}
This follows from Proposition \ref{prop:Ilowerbd}. From \eqref{eqn:JintermsofI}, 
\[J(1-x)=\sqrt{2} w_0 I(x w_0^{-1})\]
and $w_0\in [0.895, 0.896]$. 
Thus
\begin{equation}\label{eqn:Jlowerbd}
J(1-x) \ge 2 x \sqrt{\log(w_0/x)} (1- \varepsilon),
\end{equation}
where
\[  \varepsilon = \tfrac12 \tfrac{\log\log(w_0/x)}{\log(w_0/x)} + \tfrac{\log(2\pi^{1/2})}{\log(w_0/x)}. \]
Finally, if $x\le \frac1{64}$, then $2 (1-\varepsilon)> 1$
(by evaluating this decreasing function at $x=\frac1{64}$).

\begin{lem}\label{lem:JvsL_fun}
For all $x\in (\frac12, 1-10^{-10^{380}})$,
\[ J(x)>L_{\beta_0}(1-x).\]
\end{lem}

\begin{proof}
We begin with the case $\frac12<x\le \frac{2047}{2048}$.
Using \eqref{eqn:DDL}, a lower bound for 
\[-\partial_x^2[J(x)-L_\beta(1-x)]\]
is
\[ \underline{g_{JL,\beta}}(\underline{x},\overline{x}) = 2 \overline{J}(\underline{x},\overline{x})^{-1} -\beta \log(2)^{-\beta} (1-\overline{x})^{-1} (1-\beta+\log\tfrac1{1-\overline{x}})(\log\tfrac{1}{1-\underline{x}})^{-2+\beta} \]
(here $\overline{J}$ is an upper enclosure for $J$, see \eqref{eqn:Jenclosure}).

Running $\textsc{Partition}_1(\underline{g_{JL,\beta_0}},[\frac12,\frac{2047}{2048}])$ shows 
\begin{equation}\label{eqn:g_JL_auto}\auto
-\partial_x^2[J(x)-L_\beta(1-x)] > 0.08,
\end{equation}
for $x\in [\frac12,\frac{2047}{2048}]$, so $J(x)-L_{\beta_0}(1-x)$ is a strictly concave function on this interval and it suffices to evaluate it at $x=\frac12$, where the value is zero and at $x=\frac{2047}{2048}$, where the value is $>10^{-4}$. 

Next we consider the case $\frac{2047}{2048}<x<1-10^{-10^{380}}$. 
Setting $s=1-x<\frac1{2048}$, \eqref{eqn:Jlowerbd} implies 
\[ J(x)=J(1-s) \ge 2 s(1-\varepsilon) \sqrt{\log(w_0/s)} \]
with $
\varepsilon = \tfrac12 \tfrac{\log\log(w_0/s)}{\log(w_0/s)} + \tfrac{\log(2\pi^{1/2})}{\log(w_0/s)}.$
Thus,
\[ J(x) - L_\beta(1-x) \ge s \Big(2(1-\varepsilon)\sqrt{\log(w_0/s)} - (\log 2)^{-\beta_0} (\log(1/s))^{\beta_0}\Big). \]
Let $u=\log(w_0/s)$.
If $s\in (10^{-10^{380}}, \frac1{2048})$, then $u\in (7, 10^{381})$.
Thus it suffices to show
\begin{equation}\label{eqn:funnylemmapf1}
2u-\log(u)-\log(4\pi) - (\log 2)^{-\beta_0} u^{\frac12}(u^{\beta_0} +\log(w_0^{-1})^{\beta_0})>0
\end{equation}
for all $u\in [7,10^{381}]$.
Here we have used $\log(1/s)=u+\log(w_0^{-1})$ and $(u+v)^{\beta_0}\le u^{\beta_0}+v^{\beta_0}$.
We first do so on the interval, say $u\in [1200, 10^{381}]$.
Evaluating the decreasing function $u\mapsto -\log(u)-\log(4\pi)$ at $u=10^{381}$ and the coefficients of other terms, the left-hand side of \eqref{eqn:funnylemmapf1} is
\[ > - 1.21 u^{1.00057}+2u -0.4 u^{0.5} - 880. \]
This is a fractional polynomial in $u$ and by Descartes' rule of signs (see e.g. \cite{Jam06}) it can have at most two roots in $(0,\infty)$. Evaluating and using the intermediate value theorem then shows that it has exactly two roots, one in the interval $u\in [1000, 1200]$ and one in the interval $u\in [10^{381}, 10^{390}]$ with positive value on the interval $u\in [1200, 10^{381}]$.
Showing that \eqref{eqn:funnylemmapf1} also holds for $u\in [7, 1200]$ follows the same argument.
\end{proof}
\begin{rem}
Here is a ``back-of-the-envelope'' version of this calculation: From the asymptotic estimates for $J$ one sees that $J(1-s)-L_\beta(s)$ must change sign roughly where
\[ 2(\log 2)^{\beta_0}=(\log(1/s))^{\beta_0-\frac12} \]
\[ s = e^{-(2(\log 2)^{\beta_0})^{1/(\beta_0-1/2)}} \approx 10^{-9.4\cdot 10^{387}} \]
Controlling the error shows the zero must lie in $[10^{-10^{388}}, 10^{-10^{387}}]$.
\end{rem}

\section{Proof of the two-point inequality}\label{sec:maintwopt}
\begin{figure}[th]
\centering
\begin{tikzpicture}[scale=9]

\draw (0,0)--(0,1)--(1,1);
\draw[opacity=.1] (0,0)--(1,0)--(1,1);
\draw (0,0) -- (1,1);
\draw (0,1/2) -- (1/2,1/2);
\draw (1/4,1/2) -- (1/4,1);
\draw (0,1) -- (1/2,1/2);
\draw (1/2,1/2) -- (1/2,1);
\node at (0,-.05) {\footnotesize $0$};
\node at (1/4,-.05) {$\tfrac{1}{4}$};
\node at (1/2,-.05) {$\tfrac{1}{2}$};
\node at (1,-.05) {\footnotesize $1$};
\node at (-.03,1/2) {$\tfrac{1}{2}$};
\node at (-.03,1) {\footnotesize $1$};
\node at (4/32,6/16) {\footnotesize  \hyperref[sec:caseQ]{$Q$}};
\node at (1/2+5/32,7/8) {\footnotesize \hyperref[sec:caseJ]{$J$}};
\node at (4/32,10/16) { \footnotesize \hyperref[sec:caseLJQ]{$LJQ$}};
\node at (4.9/16,10/16) {\footnotesize \hyperref[sec:caseQJQ]{$QJQ$}};
\node at (6.25/32,7/8) {\footnotesize  \hyperref[sec:caseLJ]{$LJ$}};
\node at (6/16,7/8) {\footnotesize \hyperref[sec:caseQJ]{$QJ$}};

\end{tikzpicture}
\caption{Case distinction in the proof of Theorem \ref{thm:maintwopt}.}
\label{fig:cases}
\end{figure}
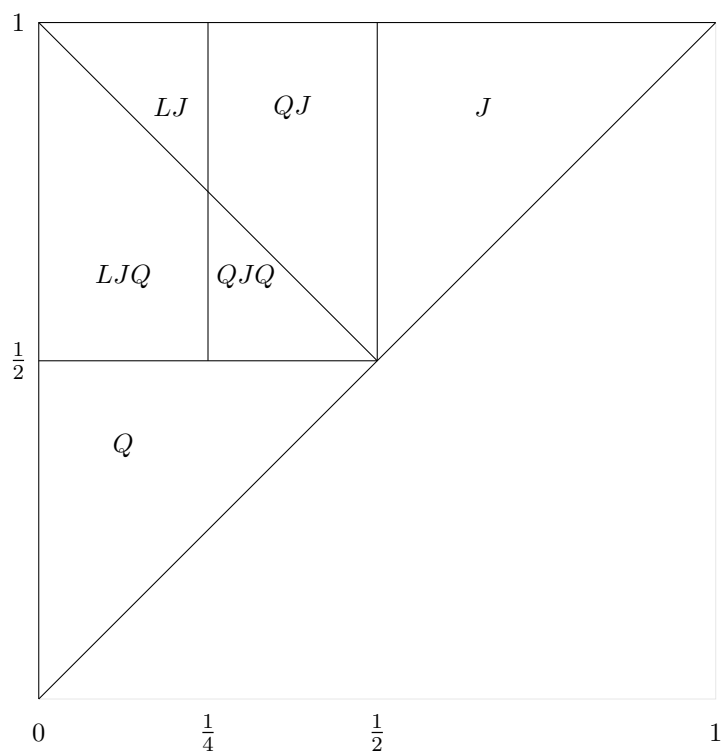

In this section we prove Theorem \ref{thm:maintwopt}. The proof is split into various subcases according to the piecewise definition of $b_\beta$ in \eqref{eqn:bbdef} (see Figure~\ref{fig:cases}).

\subsection{\texorpdfstring{Case $J$: $\frac12\le x\le y\le 1$}{Case J}}
\label{sec:caseJ}
This is the most critical case. In particular, this is where the precise values of $\beta_0$, $c_0$ originate.
Recall from Corollary \ref{cor:JsatisfiesBobkovtwopt} that for $x_0\le x\le y\le 1$ the function $J$ satisfies the best possible estimate, $G^1_{\beta}[J](x,y)\ge 0$ for $\beta=\frac12$. When $x<x_0\approx 0.552$ this estimate fails for $\beta=\frac12$.

\begin{prop}\label{prop:caseJ}
Let $\beta\ge \beta_0=0.50057$ and $c_0=0.997$. Then for $\frac12\le x\le x_0$ and $x\le y\le 1$,
\[G^1_{\beta}[J](x,y)\ge 0\quad\text{and}\quad
G^1_{\frac12}[c_0 J](x,y)\ge 0. \]
\end{prop}
By monotonicity in $\beta$, we only need to show $G_{\beta_0}^1[J]\ge 0$ to show the first part.
The proof is split into two further cases according to the size of $y$ (see Figure \ref{fig:jcrit}).
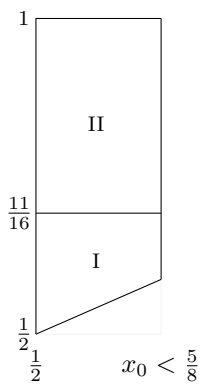
\begin{figure}[hb]
\centering
\begin{tikzpicture}[yscale=14,xscale=32]
\def\Jpicturefakeone{.8}
\draw (1/2,1/2)--(1/2,\Jpicturefakeone)--(.552,\Jpicturefakeone);
\draw (1/2,1/2)--(0.552, 0.552);
\draw (0.552,0.552)--(0.552,\Jpicturefakeone);
\draw (1/2,0.615)--(0.552,0.615);
\node at (1/2,1/2-.03) {\footnotesize $\tfrac{1}{2}$};
\node at (0.552,1/2-.03) {\footnotesize $x_0<\tfrac{5}8$};
\node at (1/2-.005,1/2) {\footnotesize $\tfrac{1}{2}$};
\node at (1/2-.007,0.615) {\footnotesize $\tfrac{11}{16}$};
\node at (1/2-.005,\Jpicturefakeone) {\tiny $1$};
\draw[opacity=.1,color=gray] (.5,.5) -- (.552, .5) -- (.552, .552);
\node at (.525, .57) {\tiny \hyperref[sec:J-I]{I}};
\node at (.525, .7) {\tiny \hyperref[sec:J-II]{II}};

\end{tikzpicture}
\caption{Critical region of Case $J$: $\frac12\le x\le x_0$}
\label{fig:jcrit}
\end{figure}

\subsubsection{Case I: $\frac12\le x\le  x_0$, $x\le y\le \frac{11}{16}$}
\label{sec:J-I}
Let $h=y-x$. 
The proof will use automatic dyadic partitioning as described in \S \ref{sec:auto}.
Since $x_0\approx 0.552$ is not a floating point number, we must account for cases when $x$ is slightly larger than $x_0$. This causes no issues. In fact, the argument is robust enough to afford a generous margin away from $x_0$ and with no additional effort we will prove the desired inequality in the larger region
\[ \tfrac12\le x\le \tfrac58, 0\le h\le \tfrac{3}{16}. \]
By Lemma \ref{lem:neardiaglowerbd},
\[ G^1_\beta[c\cdot J](x,x+h)\ge \beta c^{1-\frac1\beta} J(x+h)^{1-\frac1\beta}h^{\frac1\beta} + c(J(x) + J(x+h)-2J(x+\tfrac{h}2)) \] 
\[- \tfrac12 \beta(1-\beta) c^{1-\frac2\beta} J(x+h)^{1-\frac2\beta}h^{\frac2\beta}. \]
Taylor's theorem shows that for every $N\ge 2$:
\[ J(x) + J(x+h)-2J(x+\tfrac{h}2) = \sum_{k=2}^{N-1} \tfrac1{k!}(1-2^{-k+1}) J^{(k)}(x)h^k + E_N, \]
where 
\[ E_N = \tfrac1{N!}(J^{(N)}(\xi_1)-2^{-N+1} J^{(N)}(\xi_2)) h^N \]
and $\xi_1\in [x,x+h]$, $\xi_2\in [x,x+\frac{h}2]$ are intermediate values.
Thus,
\[ G_\beta^1[c\cdot J](x,x+h)\ge h^{\frac1\beta} g_{J,1,\beta,c}(x,h), \]
where
\[ g_{J,1,\beta,c}(x,h)= \beta c^{1-\frac1\beta}J(x+h)^{1-\frac1\beta} - \tfrac12 \beta(1-\beta)c^{1-\frac2\beta}J(x+h)^{1-\frac2\beta}h^{\frac1\beta} \]  \[ +c \sum_{k=2}^{N-1} \tfrac1{k!}(1-2^{-k+1}) J^{(k)}(x)h^{k-\frac1\beta} + c E_N h^{-\frac1\beta}. \]
To get a sufficiently accurate lower bound we will unfortunately need to use $N=6$ which makes the following expressions somewhat lengthy. From $J\cdot J''=-2$, we may calculate
\[ J^{(3)}=2J' J^{-2},\]
\begin{equation}\label{eqn:D4J}
J^{(4)}=-4 (1+|J'|^2)J^{-3}<0,
\end{equation}
\[ J^{(5)}=4J'(7+3|J'|^2)J^{-4}, \]
\[ J^{(6)}=-8(7+23|J'|^2+6|J'|^4)J^{-5}<0. \]
Recall that $J$ is increasing on $[\frac12,x_0]$ and decreasing on $[x_0,1]$ (see \eqref{eqn:JintermsofI} and \eqref{eqn:x0def}).
Thus an interval enclosure for $J$ is given by
\begin{equation}\label{eqn:Jenclosure}
\underline{J}(\underline{x},\overline{x})=\min(J(\underline{x}),J(\overline{x})),
\end{equation}
\[ \overline{J}(\underline{x},\overline{x})=\left\{
\begin{array}{ll}
J(\overline{x}), & \text{if}\;\overline{x}<x_0,\\
J(\underline{x}), & \text{if}\;\underline{x}>x_0,\\
J(x_0), & \text{else}.
\end{array}
\right.\]
Also observe $J'(x_0)=0$ and $J''<0$, so $J'$ is strictly decreasing and $|J'|$ has the interval enclosure
\[ \underline{|J'|}(\underline{x},\overline{x})= \left\{
\begin{array}{ll}
J'(\overline{x}), & \text{if}\;\overline{x}<x_0,\\
-J'(\underline{x}), & \text{if}\;\underline{x}>x_0,\\
0, & \text{else},
\end{array}
\right.\]
\[ \overline{|J'|}(\underline{x},\overline{x}) = \max(|J'(\underline{x})|, |J'(\overline{x})|). \]
The odd order derivatives $J^{(3)}$ and $J^{(5)}$ change sign at $x_0$ and admit the tight lower bounds
\[ \underline{J^{(3)}}(\underline{x},\overline{x}) = 2 J'(\overline{x}) J(\overline{x})^{-2} \mathbf{1}_{\overline{x}<x_0} - 2 \overline{|J'|}(\underline{x},\overline{x}) \underline{J}(\underline{x},\overline{x})^{-2} \mathbf{1}_{\text{not}\;\overline{x}< x_0} \]
and
\[ \underline{J^{(5)}}(\underline{x},\overline{x}) = 4 J'(\overline{x})(7+3 |J'(\overline{x})|^2)J(\overline{x})^{-4}\mathbf{1}_{\overline{x}<x_0} \] 
\[\hspace{2cm}- 4 \overline{|J'|}(\underline{x},\overline{x})(7+3\overline{|J'|}(\underline{x},\overline{x})^2)\underline{J}(\underline{x},\overline{x})^{-4}\mathbf{1}_{\text{not}\;\overline{x}< x_0}. \]
Now we can compose a tight lower bound for $g_{J,1,\beta,c}$ as
\[ \underline{g_{J,1,\beta,c}}(\underline{x},\overline{x},\underline{h},\overline{h}) = \beta c^{1-\frac1\beta} \overline{J}(\underline{x}+\underline{h},\overline{x}+\overline{h})^{1-\frac1\beta} - \tfrac12 \beta(1-\beta)c^{1-\frac2\beta}\underline{J}(\underline{x}+\underline{h},\overline{x}+\overline{h})^{1-\frac2\beta}\overline{h}^{\frac1\beta} \]  
\[ - \tfrac{c}2  \underline{J}(\underline{x},\overline{x})^{-1} \overline{h}^{2-\frac1\beta} 
+c \Big(\tfrac14 J'(\overline{x}) J(\overline{x})^{-2}\underline{h}^{3-\frac1\beta} + \tfrac1{32}J'(\overline{x})(7+3 |J'(\overline{x})|^2)J(\overline{x})^{-4}\underline{h}^{5-\frac1\beta}\Big)\mathbf{1}_{\overline{x}<x_0}
\]
\[ - c\Big(\tfrac14 \overline{|J'|}(\underline{x},\overline{x}) \underline{J}(\underline{x},\overline{x})^{-2} \overline{h}^{3-\frac1\beta} +
\tfrac1{32} \overline{|J'|}(\underline{x},\overline{x})(7+3\overline{|J'|}(\underline{x},\overline{x})^2)\underline{J}(\underline{x},\overline{x})^{-4} \overline{h}^{5-\frac1\beta}
\Big)\mathbf{1}_{\text{not}\;\overline{x}< x_0} \]
\[ -\tfrac{7c}{48}(1+\overline{|J'|}(\underline{x},\overline{x})^2)\underline{J}(\underline{x},\overline{x})^{-3}\overline{h}^{4-\frac1\beta} \]
\[ -\tfrac{c}{90}(7+23\overline{|J'|}(\underline{x},\overline{x}+\overline{h})^2 + 6 \overline{|J'|}(\underline{x},\overline{x}+\overline{h})^4)\underline{J}(\underline{x},\overline{x}+\overline{h})^{-5}\overline{h}^{6-\frac1\beta} \]
\[ + \tfrac{c}{2880}(7+23\underline{|J'|}(\underline{x},\overline{x}+\tfrac12\overline{h})^2+6 \underline{|J'|}(\underline{x},\overline{x}+\tfrac12\overline{h})^4)\overline{J}(\underline{x},\overline{x}+\tfrac12\overline{h})^{-5} \underline{h}^{6-\frac1\beta}. \]
Running $\textsc{Partition}_2(\underline{g_{J,1,\beta,c}},[\tfrac12,\tfrac58]\times [0,\tfrac3{16}])$ for each $(\beta,c)\in \{(\beta_0,1),(\tfrac12,c_0)\}$ shows that 
\begin{equation}\label{eqn:g_J_1_auto}\auto
g_{J,1,\beta,c}>10^{-7}
\end{equation}
on this rectangle. 
Figure~\ref{fig:critpartition} visualizes the admissible partition.
\begin{figure}[ht]
\includegraphics[width=7cm]{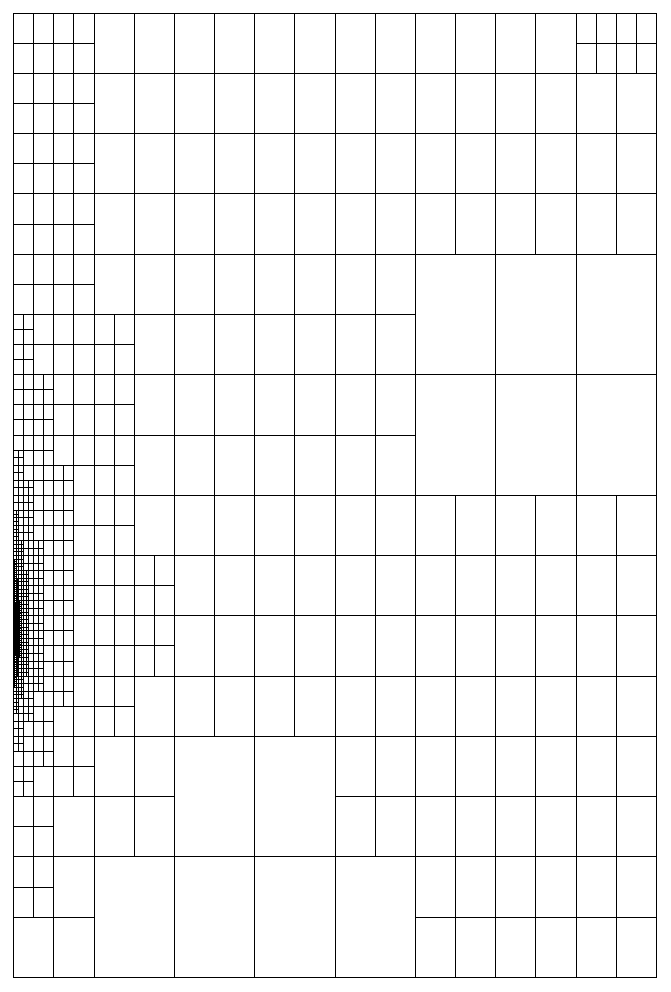}
\caption{An admissible partition of $\{~\tfrac12\le x\le \tfrac58, 0\le h\le \tfrac{3}{16}\}$ into $1303$ subrectangles as generated by $\textsc{Partition}_2$ to verify \eqref{eqn:g_J_1_auto} with $(\beta,c)=(\beta_0,1)$.}
\label{fig:critpartition}
\end{figure}

\subsubsection{Case II: $\frac12\leq x\leq x_0, \frac{11}{16}\le y\le 1$}
\label{sec:J-II}
By monotonicity it suffices to show the stronger conclusion
$G^1_{\frac12}[J](x,y)\ge 0$
or equivalently, 
\[g_{J,2}(x,y)=(y-x)^2 +J(y)^2-(2J(\tfrac{x+y}{2})-J(x))^2\ge 0.\]
The quantity $2J(\frac{x+y}{2})-J(x)\ge 0$ is decreasing in $y$ (since $\frac{x+y}2\ge \tfrac12(\tfrac12+\tfrac{11}{16})=0.59375>x_0$, so $J'(\frac{x+y}2)<0$) and also decreasing in $x$, because
the $x$-derivative is
\[ J'(\tfrac{x+y}2)-J'(x) = J''(\xi)\tfrac{y-x}2 < 0 \]
by the mean value theorem.
Therefore, a tight lower bound for $g_{J,2}$ is given by
\[ \underline{g_{J,2}}(\underline{x},\overline{x},\underline{y},\overline{y})=(\underline{y}-\overline{x})^2 +J(\overline{y})^2-(2J(\tfrac{\underline{x}+\underline{y}}{2})-J(\underline{x}))^2. \]
Running $\textsc{Partition}_2(\underline{g_{J,2}},[\frac12,\frac9{16}]\times[\frac{11}{16},1])$ shows that 
\begin{equation}\label{eqn:g_J_2_auto}\auto
g_{J,2}(x,y)>10^{-8}
\end{equation}
for $(x,y)$ in this rectangle (and note that $\frac{9}{16}=0.5625>x_0$).

\subsection{\texorpdfstring{Case $Q$: $0\le x\leq  y \leq\frac12$}{Case Q}}
\label{sec:caseQ}

\begin{prop}
For $\beta\in \{\tfrac12, \beta_0\}$ and $0\le x\le y\le \tfrac12$:
\[ G^1_{\beta}[\max(L_\beta, Q_\beta)](x,y)\ge 0 \]
\end{prop}
\begin{rem}
The argument is not sensitive to the exact value of $\beta_0$: the conclusion holds for all $\beta\in [\tfrac12,1]$, 
but we don't pursue this here.
\end{rem}
It is known that
\[ G^1_\beta[L_\beta](x,y)\ge 0 \]
holds for all $0\le x\le y\le \tfrac12$ (see \cite[Lemma 2.2]{BIM23}).
By Lemma \ref{lem:maxlemma2} and Lemma \ref{lem:LQmax} we therefore have the claim on the triangle I in Figure~\ref{fig:lq} and it now suffices to show
\begin{equation}\label{eqn:GLQge0}
G^1_\beta[Q_\beta](x,y)\ge 0
\end{equation}
for all $0\le x\le y\le \tfrac12$ with $\frac{x+y}{2}\ge \tfrac14$.
Observe that $G_\beta^1[Q_\beta](x,y)=0$ if $x=y$ and if $(x,y)=(0,\tfrac12)$.
This leads to distinguishing two further cases (regions II and III in Figure \ref{fig:lq}).

\begin{figure}[ht]
\centering
\begin{tikzpicture}[scale=8]
\draw (0,0)--(1/2,1/2)--(0,1/2)--cycle;
\draw (0,1/2)--(1/4,1/4);
\draw (1/8,3/8)--(1/4,1/2);
\node at (0,0-.05) {\tiny $0$};
\node at (1/4,-.05) {\footnotesize $\tfrac{1}{4}$};
\node at (1/2,-0.05) {\footnotesize $\tfrac{1}{2}$};
\node at (0-.03,0) {\tiny $0$};
\node at (0-.03,1/2) {\footnotesize $\tfrac{1}{2}$};
\node at (0-.03,1/4) {\footnotesize $\tfrac{1}{4}$};
\draw[opacity=.1] (0,0)--(1/2,0)--(1/2,1/2);
\draw[opacity=.1] (1/4,0)--(1/4,1/2);
\draw[opacity=.1] (0,1/4)--(1/2,1/4);
\node  at (1/8,1/4) {\tiny I};
\node at (1/4,3/8) {\tiny \hyperref[sec:Q-II]{II}};
\node  at (1/8,3.6/8) {\tiny \hyperref[sec:Q-III]{III}};
\end{tikzpicture}
\caption{Case $Q$}
\label{fig:lq}
\end{figure}

\subsubsection{Near diagonal: $0\le y-x\le \tfrac14$, $\frac{x+y}2\ge \frac14$}
\label{sec:Q-II}
This is the quadrangle II in Figure \ref{fig:lq}. Note that in this region $y\ge \frac14$.
Let $h=y-x$.
We will prove the desired bound in the larger region
\[ 0\le h\le \tfrac14\le y\le \tfrac12. \]
By Lemma \ref{lem:neardiaglowerbd},
\[ G^1_\beta[Q_\beta](x,y)\ge h^{\frac1\beta} g_{Q,1,\beta}(h,y),\]
where
\[ g_{Q,1,\beta}(h,y)= \beta Q_\beta(y)^{1-\frac1\beta} - (2Q_\beta(y-\tfrac{h}2)-Q_\beta(y-h)-Q_\beta(y))h^{-\frac1\beta} - \tfrac12 \beta (1-\beta) Q_\beta(y)^{1-\frac2\beta} h^{\frac1\beta}. \]
Calculate
\[ 2Q_\beta(y-\tfrac{h}2)-Q_\beta(y-h)-Q_\beta(y) = h^2 (\alpha_0 - 2\alpha_1 h + 4\alpha_1 y) \ge 0, \]
where $\alpha_0,\alpha_1$ are as in \eqref{eqn:alpha01def}.
Thus
\[ g_{Q,1,\beta}(h,y) = \beta Q_\beta(y)^{1-\frac1\beta} - (\alpha_0-2\alpha_1 h+4\alpha_1 y)h^{2-\frac1\beta} - \tfrac12 \beta (1-\beta) Q_\beta(y)^{1-\frac2\beta} h^{\frac1\beta} \]
If $\beta>\tfrac12$, then $g_{Q,1,\beta}(h,y)>0$ and this can be proved automatically: 
Since $Q_\beta$ is monotone increasing on $[0,\tfrac12]$ (see Lemma \ref{lem:Qbprop}), we obtain the tight lower bound
 \[ \underline{g_{Q,1,\beta}}(\underline{h},\overline{h},\underline{y},\overline{y}) = \beta Q_\beta(\overline{y})^{1-\frac1\beta} - (\alpha_0 - 2\alpha_1 \underline{h} + 4\alpha_1 \overline{y})\overline{h}^{2-\frac1\beta} \] 
\[- \tfrac12 \beta (1-\beta)Q_\beta(\underline{y})^{1-\frac2\beta} \overline{h}^{\frac1\beta}. \]
Running $\textsc{Partition}_2(\underline{g_{Q,1,\beta_0}},[0,\tfrac14]\times[\tfrac14,\tfrac12])$ gives 
\begin{equation}\label{eqn:g_Q_1_auto}\auto
g_{Q,1,\beta_0}(h,y)>0.001
\end{equation}
and thus finishes the proof of \eqref{eqn:GLQge0} near the diagonal for $\beta=\beta_0$. 
The case $\beta=\frac12$ needs 
a little more work because
$g_{Q,1,\frac12}(0,\tfrac12)=0$.
We have
\[ g_{Q,1,\frac12}(h,y) = \tfrac12 Q_{\frac12}(y)^{-1} - (\alpha_0+4\alpha_1 y) + 2\alpha_1 h - \tfrac18 Q_{\frac12}(y)^{-3} h^2\]
Observe that 
\[\partial_h^2 g_{Q,1,\frac12}(h,y)=-\tfrac14 Q_{\frac12}(y)^{-3}<0,\]
so the function is concave in $h$ and it suffices to evaluate at $h=0$ and $h=\tfrac14$:
\begin{enumerate}
\item $g_{Q,1,\frac12}(0,y)\ge 0$ for all $y\in[\tfrac14,\tfrac12]$.
\begin{proof}
We need to show
\[ g_{Q,1,\frac12}(0,y)=\tfrac12 Q_{\frac12}(y)^{-1} - (\alpha_0+4\alpha_1 y) \ge 0\]
Multiplying by $2Q_{\frac12}(y)>0$ this is 
\[ 1 - 2(\alpha_0+4\alpha_1 y)Q_{\frac12}(y) \ge 0 \]
The left-hand side is a polynomial of degree $4$ in $y$.
Changing variables $t=\tfrac12-y$, plugging in the definition of $Q_{\frac12}(\tfrac12-t)$ and simplifying this becomes
\[ 1-\tfrac13 (1-4t^2)(1-4\alpha_1 t)(3-4\alpha_1 t) \ge 0. \]
The left-hand side equals zero at $t=0$ and the three factors do not change sign on the interval $t\in [0,\tfrac14]$, so the left-hand side is increasing in $t\in [0,\tfrac14]$, which shows that $g_{Q,1,\frac12}(0,y)\ge 0$ for $y\in [\tfrac14, \tfrac12]$.
\end{proof}
\item $g_{Q,1,\frac12}(\tfrac14,y)>0$ for all $y\in [\tfrac14,\tfrac12]$
\begin{proof}
Evaluate
\begin{equation}\label{eqn:g_Q_1_y1_4_auto}\auto 
\underline{g_{Q,1,\frac12}}(\tfrac14,\tfrac14,\tfrac14,\tfrac38) > 0.01\;\;\text{and}\;\;
\underline{g_{Q,1,\frac12}}(\tfrac14,\tfrac14,\tfrac38,\tfrac12)> 0.001.
\end{equation}
\end{proof}
\end{enumerate}

\subsubsection{Far from diagonal: $\tfrac14\le y-x\le \tfrac12$, $\frac{x+y}2\ge \frac14$}
\label{sec:Q-III}
This is the triangle III in Figure \ref{fig:lq}.
Again let $h=y-x$. 
The idea is to reduce to a fractional polynomial in $h$ using logarithmic derivatives.
The desired inequality \eqref{eqn:GLQge0} can be written in equivalent form as
\[ G_{Q,\beta}(h,y)=\log(h^{\frac1\beta} + Q_\beta(y)^{\frac1\beta}) - \tfrac1\beta \log(2 Q_\beta(y - \tfrac{h}{2}) - Q_\beta(y - h))\ge 0. \]
\begin{lem}\label{lem:LQIII}
For $\beta\in \{\beta_0,\frac12\}$ and $y\in[\frac14,\frac12]$, the function
\[h\mapsto \partial_h G_{Q,\beta}(h,y)\]
vanishes at most once on the interval $[\tfrac14, \tfrac12]$. Also, $\partial_h G_{Q,\beta}(\tfrac14,y)>0$.
\end{lem}
\begin{rem}
This is slightly more than we need: it would be enough to show this for $y\in [\frac38,\frac12]$ and $h\in [\tfrac14,y]$.
\end{rem}
The proof of Lemma \ref{lem:LQIII} is postponed to the end of this section.
The lemma implies that 
$h\mapsto G_{Q,\beta}(h,y)$ must achieve its minimum on $[\frac14,y]$ at $h=\frac14$ or at $h=y$.
In the near diagonal case we have already proved that $G_{Q,\beta}(\tfrac14,y)\ge 0$ (this is the line segment with $y-x=\tfrac14$ in Figure \ref{fig:lq}).
Thus it remains to show
\[ G_{Q,\beta}(y,y)\ge 0, \]
i.e. 
\[ \log(y^{\frac1\beta} + Q_\beta(y)^{\frac1\beta}) - \tfrac1\beta \log(2 Q_\beta(\tfrac{y}{2}))\geq 0.\]
Equivalently, we need to show
\[y^{\frac1\beta} + Q_\beta(y)^{\frac1\beta} - (2 Q_\beta(\tfrac{y}{2}))^{\tfrac1\beta }\geq 0.\]
Multiplying by $y^{-\tfrac{1}{\beta}}$ and defining
\[f_\beta(y)=1+R_\beta(y)^{\frac{1}{\beta}}-R_\beta(\tfrac{y}{2})^{\frac{1}{\beta}}\]
where 
$R_\beta(y) = y^{-1}Q_\beta(y) = \tfrac23  (1-y)(2^{\beta+2}-3+4(3-2^{\beta+1})  y),$
we see that it is now equivalent to show 
$f_\beta(y) \geq 0$ 
for  all $y\in [\tfrac{1}{4},\tfrac{1}{2}]$. 

We evaluate  $f_\beta(\tfrac{1}{2})=0$.  We claim that $f_\beta' \leq 0$ on $[\tfrac{1}{4},\tfrac{1}{2}]$, which will in turn imply $f_\beta\geq 0$ on $[\tfrac{1}{4},\tfrac{1}{2}]$.
To see the claim, we calculate
\[f_\beta'(y) = \tfrac{1}{\beta}R_\beta(y)^{\frac{1}{\beta}-1}R_\beta'(y) - \tfrac{1}{2\beta}R_\beta(\tfrac{y}{2})^{\frac{1}{\beta}-1}R_\beta'(\tfrac{y}{2})\]
 To show $f_\beta'(y)\leq 0$ it is equivalent  to  show that $\beta y f_\beta'(y)\leq 0$. 
Setting $f_{\beta,0}(y)=yR_\beta(y)^{\frac{1}{\beta}-1}R_\beta'(y)$,   it holds
\[\beta y f_\beta'(y) = f_{\beta,0}(y)-f_{\beta,0}(y/2).\]
 Thus, it suffices to show that $f_{\beta,0}'\leq 0$.
Let us further denote 
\[f_{\beta,1}(y)=R_\beta(y)^{\frac{1}{\beta}-1},\]
\[f_{\beta,2}(y)=yR_\beta'(y),\]
 so that $f_{\beta,0}(y)=f_{\beta,1}(y)f_{\beta,2}(y)$. 
Since  $R_\beta\geq  0$ and $\,R_\beta'\leq  0$ on $[\tfrac{1}{4},\tfrac{1}{2}]$, also $f_{\beta,1}\geq  0$ and $f_{\beta,2}\leq  0$ on this interval. 
 
We   calculate 
\[f_{\beta,1}'(y) = (\tfrac{1}{\beta}-1)R_\beta(y)^{\frac{1}{\beta}-2}R_\beta'(y),\]
 \[f_{\beta,1}''(y) = (\tfrac{1}{\beta}-1)(\tfrac{1}{\beta}-2)R_\beta(y)^{\frac{1}{\beta}-3}(R_\beta'(y))^2+(\tfrac{1}{\beta}-1)R_\beta(y)^{\frac{1}{\beta}-2}R_\beta''(y)\]
 \[f_{\beta,1}'''(y) = (\tfrac{1}{\beta}-1)(\tfrac{1}{\beta}-2)\big((\tfrac{1}{\beta}-3)R_\beta(y)^{\frac{1}{\beta}-4}(R_\beta'(y))^3 + 3 R_\beta(y)^{\frac{1}{\beta}-3}R_\beta'(y)R_\beta''(y)\big)\]
Since $\beta\in \{\beta_0,\tfrac{1}{2}\}$, one has $R_\beta''\leq 0$ and $\tfrac{1}{\beta}>1,\tfrac{1}{\beta}\leq 2,\tfrac{1}{\beta}<3$, which in turn implies  $f_{\beta,1}'\leq 0$,   $f_{\beta,1}''\leq 0$,  and $f_{\beta,1}'''\leq  0$. To calculate the third derivative we also used that   $R_\beta'''=0$.

Next we compute
\[f_{\beta,2}'(y)=R_\beta'(y)+yR_\beta''(y),\]
\[f_{\beta,2}''(y) = 2R_\beta''(y) +yR_\beta'''(y) = 2R_\beta''(y) \]
Thus,   $f_{\beta,2}'\leq  0$ and $f_{\beta,2}''\leq  0$. 

Using $f_{\beta,2}'''=0$ we calculate
\[f_{\beta,0}'''(y) = f_{\beta,1}'''(y)f_{\beta,2}(y) + 3f_{\beta,1}''(y)f_{\beta,2}'(y) + 3f_{\beta,1}'(y)f_{\beta,2}''(y)\]
The calculations above show that $f_{\beta,0}'''\geq 0$. We evaluate 
\begin{equation}
f_{\beta,0}''(\tfrac{1}{4}) >  3.2,
\end{equation}   which implies $f_{\beta,0}''>0$. We also evaluate 
\begin{equation}
f_{\beta,0}'(\tfrac{1}{2})< -0.6, 
\end{equation}
 which then shows $f_{\beta,0}'<0$ on $[\tfrac{1}{4},\tfrac{1}{2}]$,  as desired.
 
\begin{proof}[Proof of Lemma \ref{lem:LQIII}]
Here $h,y\in [\frac14,\frac12]$.
The function $\beta\partial_h G_{Q,\beta}(h,y)$ is given by
\[ \frac{h^{\frac1\beta-1}}{h^{\frac1\beta}+Q_\beta(y)^{\frac1\beta}} - \frac{Q_\beta'(y-h)-Q_\beta'(y-\frac{h}2)}{2Q_\beta(y-\frac{h}2)-Q_\beta(y-h)} \]
Observe that $2Q_\beta(y-\frac{h}2)-Q_\beta(y-h)>0$ by strict concavity of $Q_\beta$ (see Lemma \ref{lem:Qbprop}), 
so we may multiply by the common denominator $(h^{\frac1\beta}+Q_\beta(y)^{\frac1\beta})(2Q_\beta(y-\frac{h}2)-Q_\beta(y-h))$
to arrive at the quantity
\begin{equation}\label{eqn:penultLQIII} h^{\frac1\beta-1}(2Q_\beta(y-\tfrac{h}2)-Q_\beta(y-h))
\end{equation}
 \[-(Q_\beta'(y-h)-Q_\beta'(y-\tfrac{h}2))(h^{\frac1\beta}+Q_\beta(y)^{\frac1\beta}). \]
This is a fractional polynomial in $h$. To see this, calculate
\[ 2Q_\beta(y-\tfrac{h}2)-Q_\beta(y-h) 
=-2\alpha_1 h^3+(\alpha_0+4\alpha_1 y)h^2 + Q_\beta(y)
\]
\[ Q_\beta'(y-h)-Q_\beta'(y-\tfrac{h}2)=-6\alpha_1 h^2+(2\alpha_0+8\alpha_1 y)h, \]
where $\alpha_0,\alpha_1$ are as in \eqref{eqn:alpha01def}.
Thus \eqref{eqn:penultLQIII} equals
$h^{\frac1\beta-1} p_{y,\beta}(h)$, where
\[ p_{y,\beta}(h) = 4\alpha_1 h^{3} - (\alpha_0 + 4\alpha_1 y)h^{2} + 6\alpha_1 Q_\beta(y)^{\frac1\beta} h^{3-\frac1\beta}\]\[ - (2\alpha_0+8\alpha_1 y)Q_\beta(y)^{\frac1\beta} h^{2-\frac1\beta} + Q_\beta(y).\]
It will now suffice to show that $p_{y,\beta}$ is strictly decreasing and thus has at most one zero.
To do this we compute the derivative:
\[ g_{Q,2,\beta}(h,y) = -\partial_h p_{y,\beta}(h) = -a_2 h^{2} +  a_{1,y} h - a_{0^+,y} h^{2-\frac1\beta} + a_{-1^+,y} h^{1-\frac1\beta}
\]
where the coefficients can be read off from the definition of $p_{y,\beta}(h)$.
A tight lower bound of $g_{Q,2,\beta}$ is given by
\[ \underline{g_{Q,2,\beta}}(\underline{h},\overline{h},\underline{y},\overline{y}) = -12\alpha_1 \overline{h}^{2} + (2\alpha_0 + 8\alpha_1 \underline{y})\underline{h} - 6(3-\tfrac1\beta)\alpha_1 Q_\beta(\overline{y})^{\frac1\beta} \overline{h}^{2-\frac1\beta} \] 
\[ + (2-\tfrac1\beta)(2\alpha_0+8\alpha_1 \underline{y})Q_\beta(\underline{y})^{\frac1\beta} \overline{h}^{-\frac1\beta+1}. \]
Finally, running $\textsc{Partition}_2(\underline{g_{Q,2,\beta}}, [\tfrac14,\tfrac12]^2)$ for $\beta\in\{\frac12,\beta_0\}$ shows that 
\begin{equation}\label{eqn:g_Q_2_auto}\auto
g_{Q,2,\beta}(h,y)>0.01
\end{equation}
for all $(h,y)\in [\tfrac14,\tfrac12]^2$.
\end{proof}

\subsection{\texorpdfstring{Case $LJQ$: $0\leq x \leq \tfrac{1}{4},\,\tfrac{1}{2}\leq y \leq 1,x+y\leq 1$.}{Case LJQ}} 
\label{sec:caseLJQ}
We distinguish further subcases.
\begin{figure}[htb]
\centering
\begin{tikzpicture}[scale=8]
\draw (0,1/2)--(1/4,1/2)--(1/4,3/4)--(0,1)--cycle;
\draw (1/16,3/4)--(1/4,3/4);
\draw (1/16,1/2)--(1/16,3/4);
\node at (0,1/2-.05) {\tiny $0$};
\node at (1/16,1/2-.05) {\footnotesize $\tfrac{1}{16}$};
\node at (1/4,1/2-.05) {\footnotesize $\tfrac{1}{4}$};
\node at (0-.03,1/2) {\footnotesize $\tfrac{1}{2}$};
\node at (0-.03,3/4) {\footnotesize $\tfrac{3}{4}$};
\node at (0-.03,1) {\footnotesize $1$};
\draw[opacity=.1] (1/4,1/2)--(1/4,1)--(0,1);
\node  at (5/32,2.5/4) {\tiny \hyperref[sec:LJQ-I]{I}};
\node  at (1/16,13.5/16) {\tiny \hyperref[sec:LJQ-II]{II}};
\end{tikzpicture}
\caption{Case $LJQ$}
\label{fig:ljq}
\end{figure}
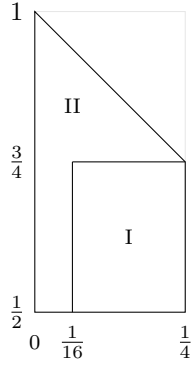

\subsubsection{Case I: $\tfrac{1}{16}\leq x\leq \tfrac{1}{4}, \tfrac{1}{2}\leq y\leq \tfrac34$}
\label{sec:LJQ-I}
This is the rectangle I in Figure \ref{fig:ljq}. Here we have
$G_\beta[b_\beta](x,y)>0$ for all $\beta\in[\frac12,1]$, but we will only prove it for $\beta\in\{\frac12,\beta_0\}$.
A naive tight lower bound for $G_\beta[b_\beta](x,y)$ suffices:
\[ \underline{g_{LJQ,1,\beta}}(\underline{x},\overline{x},\underline{y},\overline{y})=
\max(((\underline{y}-\overline{x})^{\frac1\beta}+\underline{J}(\underline{y},\overline{y})^{\frac1\beta})^\beta, \underline{y}-\overline{x}+(2^\beta-1) \underline{J}(\underline{y},\overline{y})) \] 
\[ + L(\underline{x},\beta) - 2 Q_\beta(\tfrac{\overline{x}+\overline{y}}{2}).
\]
Running $\textsc{Partition}_2(\underline{g_{LJQ,1,1/2}},[\frac1{16},\frac14]\times[\frac12,\frac34])$ shows
\begin{equation}\label{eqn:g_LJQ_1_auto}\auto
G_\beta[b_\beta](x,y)>10^{-7}
\end{equation}
for $\beta\in\{\beta_0,\frac12\}$ and $(x,y)\in [\frac1{16},\frac14]\times[\frac12,\frac34]$.

\subsubsection{Case II}
\label{sec:LJQ-II}
This is region II in Figure \ref{fig:ljq} and covers the remainder of Case $LJQ$.

\begin{prop}\label{prop:caseLJQ} Let $\beta\in[\frac12,1]$,  $x\in [0,\frac14]$ and $y\in [\frac12,1]$.
If in addition $x\le \frac1{16}$ or $y\ge \frac34$ holds, then 
\[ \label{caseljq1} G_{LJQ,\beta}(x,y)=y-x + (2^\beta-1) J(y) +L_\beta(x)-2Q_\beta(\tfrac{x+y}{2})\geq 0.\] 
\end{prop}

The proof rests on the following observation.

\begin{lem}\label{lem:LJQconcave}
For every $\beta\in [\frac12,1]$ and $y\in [0, 1]$ the function $x\mapsto G_{LJQ,\beta}(x,y)$ is strictly concave on $[0,\frac14]$.
\end{lem}
\begin{proof}
Compute 
\[ \partial_x^2 G_{LJQ,\beta}(x,y) = L''_\beta(x) - \tfrac12 Q_\beta''(\tfrac{x+y}2) \]
By Lemma \ref{lem:Qbprop} this equals
\begin{equation}\label{eqn:LJQlempf1} L''_\beta(x) + 2\alpha_0 + 4\alpha_1 (x+y)
\end{equation}
Recall that $\alpha_1=3-2^{1+\beta}>0$ iff $\beta<\log_2(3/2)$. 
Let us only consider the case $\beta<\log_2(3/2)$ -- the argument for the other (less interesting) case is analogous.
By Lemma \ref{lem:Lprop} and since $x\le \frac14< e^{-\sqrt{3}/2}$, the quantity \eqref{eqn:LJQlempf1} is increasing in both $x$ and $y$. Thus it suffices to evaluate 
\[\partial_x^2 G_{LJQ,\beta}(\tfrac14,1)=L''_\beta(\tfrac14) - \tfrac12 Q_\beta''(\tfrac58)\le L''_{\frac12}(\tfrac14) -\tfrac12 Q''_{\frac12}(\tfrac58)<-0.6, \]
where we have used that $\beta\mapsto L''_\beta(\tfrac14)$ is decreasing by Lemma \ref{lem:Lprop} and that for $x\in[\tfrac12,1]$, the function $\beta\mapsto Q_\beta''(x)$ is increasing.
\end{proof}

Lemma \ref{lem:LJQconcave} reduces the proof of Proposition \ref{prop:caseLJQ} to checking the claim on the $x$-boundary of the region II. Thus it remains to verify the following three claims:
\begin{enumerate}
\item $G_{LJQ,\beta}(0,y)\ge 0$ for all $y\in [\frac12,1]$
\item $G_{LJQ,\beta}(\tfrac1{16},y)\ge 0$ for all $y\in[\tfrac12,\tfrac34]$
\item $G_{LJQ,\beta}(1-y,y)\ge 0$ for all $y\in [\tfrac34,1]$
\end{enumerate}

\begin{proof}[Proof of (1)]
Let 
\[ f_\beta(y)=G_{LJQ,\beta}(0,y) = y + (2^\beta-1) J(y) - 2Q_\beta(\tfrac{y}2). \]
Differentiating in $\beta$ we see that $\partial_\beta f_\beta(y)$ takes the form $2^\beta \widetilde{f}(y)$ for some function $\widetilde{f}(y)$. Thus, for a fixed $y$, the quantity $f_\beta(y)$ is either monotone increasing or monotone decreasing in $\beta$. Thus, for all $\beta\in[\frac12, 1]$,
\[ f_\beta(y)\ge \min(f_{\frac12}(y), f_1(y)). \]
Thus we may assume $\beta\in\{\tfrac{1}{2},1\}$ without loss of generality. We give details only for the case $\beta=\frac12$, the case $\beta=1$ being very similar.
Using $JJ''=-2$ we see
\[ f_{\frac12}'(y) = 1 + (\sqrt{2}-1)J'(y) - Q_{\frac12}'(\tfrac{y}2), \]
\[f_{\frac12}''(y) =  -2 (\sqrt{2}-1)  J(y)^{-1} -\tfrac12 Q''_{\frac12}(\tfrac{y}2)\]
By \eqref{eqn:D4J}, $f_{\frac12}^{(4)}(y)<0$
so $f_{\frac12}''$ is  strictly concave on $(\frac12,1)$. Since $f_{\frac12}''(\tfrac{1}{2})=0$ this implies that $f_{\frac12}''$ has at most one  zero on $(\tfrac{1}{2},1)$.   We evaluate 
\begin{equation}
f_{\frac12}''(\tfrac{9}{16})>0.05,\quad
f_{\frac12}''(\tfrac{4}{5})<-0.3.
\end{equation} 
Thus, $f_{\frac12}''$  has  exactly one zero $y_0$ in  $(\tfrac{1}{2},1)$ where $f_{\frac12}'$
achieves its unique maximum.  Since $f_{\frac12}'$  is increasing on $(\tfrac{1}{2},y_0)$ and decreasing on $(y_0,1)$, and
\begin{equation}
f_{\frac12}'(\tfrac{9}{16})>0.05,\quad
f_{\frac12}'(\tfrac{15}{16})<-0.1,
\end{equation}  
it follows  that   $f_{\frac12}'$ has exactly one zero $y_1$ on $(\tfrac{1}{2},1)$. Therefore,   $f_{\frac12}$ is increasing on $[\frac12,y_1]$ and decreasing on $[y_1,1]$.  Together with   $f_{\frac12}(1)=f_{\frac12}(\tfrac{1}{2})=0$, this  implies $f_{\frac12}(y)\geq 0$ for $y\in [\frac12,1]$.
\end{proof}

\begin{proof}[Proof of (2)]
A tight lower bound is given by
\[  \underline{g_{LJQ,2}}(\underline{y},\overline{y},\underline{\beta},\overline{\beta}) = \underline{y}-\tfrac{1}{16} + (2^{\underline{\beta}}-1) \underline{J}(\underline{y},\overline{y}) +L_{\underline{\beta}}(\tfrac{1}{16})-2Q_{\overline{\beta}}(\tfrac{1}{32}+ \tfrac{\overline{y}}{2})\]
Running  
$\textsc{Partition}_2(\underline{g_{LJQ,2}},[\tfrac{1}{2},\tfrac{3}{4}]\times [\tfrac12,1])$ shows that
\begin{equation}\label{eqn:g_LJQ_2_auto}\auto
G_{LJQ,\beta}(\tfrac1{16},y)>10^{-5}
\end{equation}
for all $y\in [\frac12,\frac34]$ and all $\beta\in [\frac12,1]$.
\end{proof}

\begin{proof}[Proof of (3)]
We need to show that 
\[f_\beta(y)=G_{LJQ,\beta}(1-y,y) = 2y-1+(2^\beta-1) J(y)+L_\beta(1-y)-1\ge 0.\]
Observe that this is increasing in $\beta$, so it suffices to consider $\beta=\frac12$.
The function $f_{\frac12}$ is concave on $[\tfrac34,1]$, because
\[f_{\frac12}''(y)= (\sqrt{2}-1) J''(y)+L_{\frac12}''(1-y)\le 0\]
(recall $J''\leq 0$ and Lemma \ref{lem:Lprop}). 
Thus it suffices to evaluate at the endpoints:  $f_\frac12(1)=0$ and
\begin{equation}
f_\frac12 (\tfrac{3}{4}) > 0.01.
\end{equation}
\end{proof}

\subsection{\texorpdfstring{Case $LJ$: $0 \leq x \leq \tfrac{1}{4}, \tfrac{1}{2}\leq y\leq 1, 1\leq x+y$}{Case LJ}}
\label{sec:caseLJ}
Note that in this region we have $y\in [\tfrac{3}{4},1]$. See Figure \ref{fig:cases}.
\begin{prop}
    For $(x,y)\in [0,\tfrac{1}{4}]\times [\tfrac{3}{4},1]$, $1\leq  {x+y}$,  and  $\beta\in~[\tfrac{1}{2},1]$,  
    \[y-x + (2^\beta-1) J(y) + L_\beta(x) -2J(\tfrac{x+y}{2})\geq 0.\]
\end{prop}
\begin{proof}
The left-hand side is increasing in $\beta$, so it suffices to prove the claim for $\beta=\tfrac{1}{2}$. 
    Denote 
    \[g_{LJ}(x,y) = y-x + (\sqrt{2}-1) J(y) + L_{\frac{1}{2}}(x) -2J(\tfrac{x+y}{2})\]
We claim that  if $(x,y)\in (0,\tfrac{1}{4}]\times [\tfrac{3}{4},1]$ and $1\leq x+y$, 
then  \[\partial^2_xg_{LJ}(x,y)\leq 0.\]
Using  $JJ''=-2$ we have
 \[\partial^2_xg_{LJ} (x,y) =  
 L_{\frac{1}{2}}''(x) +J(\tfrac{x+y}{2})^{-1}.\]
Lemma \ref{lem:Lprop}  shows  $L''_{1/2}$ is increasing and negative. 
Furthermore, 
\begin{equation}
L''_{\tfrac{1}{2}}(\tfrac{1}{4})\leq -2.7. 
\end{equation}
If $\tfrac{x+y}{2}\in [\tfrac{1}{2},1-\tfrac{w_0}{2}]$, then $ J(\tfrac{x+y}{2})^{-1}\leq 2$ and hence $\partial^2_xg_{LJ} (x,y)\leq 0$. 

If $\tfrac{x+y}{2}\geq  1-\tfrac{w_0}{2}$, it suffices to show that 
$\partial^3_xg_{LJ}(x,y) \geq 0$ and   
  $\partial^2_xg_{LJ} (\tfrac{1}{4},y) \leq 0 $.
  To show the latter we first evaluate 
\begin{equation}
\partial^2_xg_{LJ} (\tfrac{1}{4},1) \leq -0.7.     
\end{equation}
  Since $J^{-1}$ is increasing on $[\tfrac{3}{4},1]$, we have $\partial^2_xg_{LJ} (\tfrac{1}{4},y) \leq 0 $  for all $y\in [\tfrac{3}{4},1]$. 
To show $\partial^3_xg_{LJ}\geq 0$ we   calculate
  \[\partial^3_xg_{LJ}(x,y) =  L_{\frac{1}{2}}'''(x) -\tfrac{1}{2}J'(\tfrac{x+y}{2}) J(\tfrac{x+y}{2})^{-2}.\]
  Since in the current region $J'(\tfrac{x+y}{2})\leq 0$ and $L'''_{1/2}\ge 0$ by Lemma \ref{lem:Lprop}, the claim follows.

It remains to show non-negativity on the boundary of the region LJ. If $y=1$, we already know $\tfrac{d^2}{dx^2}g_{LJ}(x,1)\leq 0$. Then we evaluate $g_{LJ}(0,1)=0$ and 
\begin{equation}
g_{LJ}(\tfrac{1}{4},1)\geq 0.1,
\end{equation}
 which shows $g_{LJ}(x,1)\geq 0$ for all $x\in [0,\tfrac{1}{4}]$.

Next we consider the case $x+y=1$, $x\in [0,\tfrac{1}{4}], y\in [\tfrac{3}{4},1]$ and let
\[g_{LJ}(1-y,y) = 2y+(\sqrt{2}-1) J(y)+L_{\frac{1}{2}}(1-y)-2.\]
We have 
\begin{equation}
g_{LJ}(\tfrac{1}{4},\tfrac{3}{4})\geq 0.02,
\end{equation}
$g_{LJ}(0,1)= 0$, and 
    \[\partial^2_y (g_{LJ}(1-y,y)) = (\sqrt{2}-1) J''(y) + L_{\frac{1}{2}}''(1-y)\leq 0,\]
which implies $g_{LJ}(1-y,y)\geq 0$ for all $y\in [\tfrac{3}{4},1]$.

Finally, we tackle the case $x=\tfrac{1}{4}$, $y\in [\tfrac{3}{4},1]$. We have \[ f(y)=g_{LJ}(\tfrac{1}{4},y)=y-\tfrac{1}{4}+(\sqrt{2}-1) J(y)+L_{\frac{1}{2}}(\tfrac{1}{4})-2J(\tfrac{y}{2} + \tfrac{1}{8}). \] 
Since we already know $f\ge 0$ at the endpoints, it suffices to rule out an interior minimum on $[\tfrac{3}{4},1]$. Using $JJ''=-2$ we calculate \[ f''(y) = (\sqrt{2}-1) J''(y) - \tfrac{1}{2}J''(\tfrac{y}{2} + \tfrac{1}{8}) = -2(\sqrt{2}-1) J(y)^{-1}+J(\tfrac{y}{2} + \tfrac{1}{8})^{-1}. \] 
Write $t(y)=\tfrac{y}{2}+\tfrac{1}{8}$ and $r(y)=\frac{J(y)}{J(t(y))}$. 
Then \[ f''(y)\le 0 \quad\Longleftrightarrow\quad r(y)\le 2(\sqrt{2}-1). \] 
We claim that $r$ is decreasing on $[\tfrac{3}{4},1)$. 
Indeed, \[ r'(y)=\frac{J'(y)J(t(y))-\tfrac{1}{2}J(y)J'(t(y))}{J(t(y))^2}, \] so it is enough to show \[ \frac{2J'(y)}{J(y)} \le \frac{J'(t(y))}{J(t(y))}. \] 
Let $h(u)=\frac{J'(u)}{J(u)}$. Since $JJ''=-2$, \[ h'(u)=\frac{J''(u)J(u)-(J'(u))^2}{J(u)^2}=\frac{-2-(J'(u))^2}{J(u)^2}\le 0, \] hence $h$ is decreasing. 
As $y\ge t(y)$ we have $h(y)\le h(t(y))$, and since $J'(y)<0$ and $J(y)>0$ on $[\tfrac{3}{4},1)$ we also have $h(y)\le 0$, hence $2h(y)\le h(t(y))$, proving $r'(y)\le 0$. 
Therefore $r(y)-2(\sqrt2-1)$ is decreasing, so $f''$ changes sign at most once: there exists $y_0\in[\tfrac{3}{4},1)$ such that $f''\ge 0$ on $[\tfrac{3}{4},y_0]$ and $f''\le 0$ on $[y_0,1)$, and thus $f'$ is increasing then decreasing. 
Moreover, \[ f'(y)=1+(\sqrt{2}-1)J'(y)-J'(t(y)), \qquad f'(\tfrac{3}{4})=1+(\sqrt{2}-1)J'(\tfrac{3}{4})-J'(\tfrac{1}{2}). \] Using \eqref{eqn:JintermsofI} and $w_0\in[0.895,0.896]$, one checks that $f'(\tfrac{3}{4})>0$. 
Hence $f$ has no interior local minimum on $[\tfrac{3}{4},1]$, so $\min f$ is attained at an endpoint, and we conclude $f(y)\ge 0$ for all $y\in[\tfrac{3}{4},1]$.

\end{proof}

\subsection{\texorpdfstring{Case $QJQ$: $\frac14\leq x\leq \frac12\leq y, x+y\leq 1$}{Case QJQ}}
\label{sec:caseQJQ}
Note that the region is contained in the rectangle $[\frac14,\frac12]\times [\frac12,\frac34]$, see Figure \ref{fig:cases}.

\begin{prop}[$QJQ$]
Let $\beta_1=\tfrac12+\tfrac{31}{1024}\approx 0.53$.
For all $(x,y)\in [\frac14,\frac12]\times [\frac12,\frac34]$ and $\beta\in [\tfrac12, \beta_1]$:
\begin{equation}\label{eqn:QJQclaim}
G_{QJQ,\beta}(x,y)=(y-x)^2 + J(y)^2 - (2Q_\beta(\tfrac{x+y}{2})-Q_\beta(x))^2 \geq 0.
\end{equation}
\end{prop}
\begin{rem}
The conclusion fails for $\beta\ge \beta_1+\tfrac1{1024}$. 
\end{rem}

From \eqref{eqn:Qbetaderiv} we see that for each fixed $(x,y)$ the left-hand side in \eqref{eqn:QJQclaim} is either monotone increasing or decreasing in $\beta$. 
Thus,
\[ G_{QJQ,\beta}(x,y)\ge \min(G_{QJQ,1/2}(x,y),G_{QJQ,\beta_1}(x,y)) \]
and it suffices to show the claim for $\beta=\tfrac12$ and $\beta=\beta_1$.
The $y$-derivative of the left-hand side in \eqref{eqn:QJQclaim} is $2$ times
\begin{equation}\label{eqn:gQJQdef}
g_{QJQ,\beta}(x,y)=(y-x) +J(y)J'(y)-(2Q_\beta(\tfrac{x+y}{2})-Q_\beta(x))Q_\beta'(\tfrac{x+y}{2}).
\end{equation}
We begin by showing that this quantity is strictly positive for all $(x,y)\in [\frac14,\frac12]\times [\frac12,\frac34]$ thus reducing to the case $y=\frac12$.
This can be done by $\textsc{Partition}_2$. In order to formulate a tight lower bound we record the monotonicity of the various terms appearing in \eqref{eqn:gQJQdef}:

\begin{enumerate}
\item The function $x\mapsto J(x)J'(x)$ is decreasing on $x\in [\frac12, \frac34]$.
\begin{proof}
\[(JJ')' = (J')^2-2\]
using that $JJ''=-2$.
Also $(J')^2$ is convex (Lemma \ref{lem:Iprimesqconvex}), so it suffices to evaluate $(J')^2-2$ at the endpoints $x=\frac12$ and $x=\frac34$ which shows $(J')^2-2<-1<0$.
\end{proof}
\item The quantity $2Q_\beta(\frac{x+y}2)-Q_\beta(x)$ is positive, increasing in $y$ and decreasing in $x$.
\begin{proof}
Recall Lemma \ref{lem:Qbprop}. First,
$2Q_\beta(\frac{x+y}2)-Q_\beta(x)\ge Q_\beta(y)>0$ follows since $Q_\beta$ is concave.
The quantity is increasing in $y$ since $Q_\beta'>0$ on $[0,\frac12]$. Finally, the $x$-derivative is
\[ Q'_\beta(\tfrac{x+y}{2})-Q'_\beta(x) = Q_\beta''(\xi) \tfrac{y-x}{2}<0 \]
where $\xi$ is a value in $[x,(x+y)/2]$.
\end{proof}
\item The function $x\mapsto Q_\beta'(x)$ is decreasing and positive on $[0,\frac12]$ by Lemma \ref{lem:Qbprop}.
\end{enumerate}

Therefore, a tight lower bound of $g_{QJQ,\beta}$ is given by
\[ \underline{g_{QJQ,\beta}}(\underline{x},\overline{x},\underline{y},\overline{y}) = \underline{y}-\overline{x} + J(\overline{y})J'(\overline{y}) - (2Q_\beta(\tfrac{\underline{x}+\overline{y}}2)-Q_\beta(\underline{x}))Q'_\beta(\tfrac{\underline{x}+\underline{y}}2). \]
Calling $\textsc{Partition}_2(\underline{g_{QJQ,\beta}},[\frac14,\frac12]\times[\frac12,\frac34])$ for $\beta\in\{\frac12,\beta_1\}$ 
shows that 
\begin{equation}\label{eqn:g_QJQ_auto}\auto g_{QJQ,\beta}>10^{-5}
\end{equation}
on $[\frac14,\frac12]\times [\frac12,\frac34]$. 
To finish the proof it now suffices to show \eqref{eqn:QJQclaim} for $y=\frac12$, that is
\[(\tfrac{1}{2}-x)^2 + \tfrac{1}{4}  - (2Q_\beta(\tfrac{x}{2} + \tfrac{1}{4})-Q_\beta(x))^2 \ge 0 \]
for $x\in [\frac14,\frac12]$.
The left-hand side is a polynomial in $x$ that factors as
\[ \tfrac14(1-2x)^3 p_\beta(x), \]
where
\[ p_\beta(x)= (18- 3\cdot 2^{3 + \beta}+ 
   2^{3 + 2 \beta})x^3 + (2^{5 + \beta}-3\cdot 2^{2 + 2 \beta}- 21)x^2\]\[ + (8+3\cdot 2^{1 + 2 \beta}-7\cdot  2^{1 + \beta})x + 2 - 2^{2 \beta}  \]
Plugging in $\beta=\beta_1,\frac12$ one sees that the coefficients of $x,x^2,x^3$ are positive, so $p_\beta$ is an increasing function. Finally, at $x=\frac14$ one computes 
\[p_\beta(\tfrac14)>10^{-4}>0\]
for $\beta\in \{\beta_1,\frac12\}$ (actually $p_\beta(\tfrac14)$ is decreasing in $\beta$) .

\subsection{\texorpdfstring{Case $QJ$: $\frac14\le x\le \frac12\le y\le 1$, $x+y\ge 1$}{Case QJ}}
\label{sec:caseQJ}
We distinguish two cases, see Figure \ref{fig:qj}.
The near-diagonal triangle I is the most critical: here it is again necessary to move away from $\beta=\frac12$ (or include a constant $c<1$).
Note that in this region $G^1_\beta[J](x,y)$ is increasing in $\beta$.

\begin{figure}[htb]
\centering
\begin{tikzpicture}[scale=8]
\draw (1/2,1/2)--(1/2,1)--(1/4,1)--(1/4,3/4)--cycle;
\draw (3/8,5/8)--(1/2,5/8);
\node at (1/4,1/2-.05) {\footnotesize $\tfrac{1}{4}$};
\node at (3/8,1/2-.05) {\footnotesize $\tfrac{3}{8}$};
\node at (1/2,1/2-.05) {\footnotesize $\tfrac{1}{2}$};
\node at (1/4-.03,1/2) {\footnotesize $\tfrac{1}{2}$};
\node at (1/4-.03,3/4) {\footnotesize $\tfrac{3}{4}$};
\node at (1/4-.03,5/8) {\footnotesize $\tfrac{5}{8}$};
\node at (1/4-.03,1) {\tiny $1$};
\draw[opacity=.1] (1/4,3/4)--(1/4,1/2)--(1/2,1/2);
\draw[opacity=.1] (3/8,1/2)--(3/8,5/8)--(1/4,5/8);
\node  at (3.6/8,4.75/8) {\tiny \hyperref[sec:QJ-I]{I}};
\node  at (3/8,6.5/8) {\tiny \hyperref[sec:QJ-II]{II}};
\end{tikzpicture}
\caption{Case $QJ$}
\label{fig:qj}
\end{figure}

\subsubsection{Near diagonal: $3/8\leq x\leq 1/2, 1/2\leq y\leq 5/8, x+y\geq 1$}
This is the triangle I in Figure \ref{fig:qj}.
\label{sec:QJ-I}

\begin{prop}
For $\frac38\le x\le \frac12\le y\le \frac58$ and $\beta\in [\beta_0,1]$, 
\[ G^1_{\beta}[b_{\beta_0}](x,y)\ge 0\quad\text{and}\quad G_{\frac12}^1[c_0\cdot b_{\frac12}](x,y)\ge 0. \]
\end{prop}
By monotonicity in $\beta$, it suffices to consider $\beta=\beta_0$ to show the first part of the claim.
Let us consider the equivalent expression
\[ g_{QJ,\beta,c}(x,y)=(y-x)^{\frac1\beta} +c^{\frac1\beta} J(y)^{\frac1\beta}-c^{\frac1\beta}(2 J(\tfrac{x+y}{2})-Q_\beta(x))^{\frac1\beta}.\]
(Observe that $2J(\tfrac{x + y}2) - Q_\beta(x)>0$.)  
We claim that $\partial_x g_{QJ,\beta,c}(x,y)\leq 0$ for each $y$. Calculate
\[-\beta \partial_x g_{QJ,\beta,c}(x,y) = (y-x)^{\frac{1}{\beta}-1}+c^{\frac1\beta}(2J(\tfrac{x+y}{2})-Q_\beta(x))^{\frac1\beta-1} (J'(\tfrac{x+y}{2})-Q'_\beta(x))\]
It suffices to show that this is positive.
Recalling Lemma \ref{lem:Qbprop} and the enclosures for $J, |J'|$ (see \eqref{eqn:Jenclosure}), a tight lower bound of this expression can be given by
\[ \underline{g_{QJ,1,\beta,c}}(\underline{x},\overline{x},\underline{y},\overline{y}) 
=(\underline{y}-\overline{x})^{\frac1\beta-1}+c^{\frac1\beta}(2\underline{J}(\tfrac{\underline{x}+\underline{y}}2,\tfrac{\overline{x}+\overline{y}}2)-Q_\beta(\overline{x}))^{\tfrac1\beta-1} J'(\tfrac{\overline{x}+\overline{y}}2)\mathbf{1}_{\frac{\overline{x}+\overline{y}}{2}<x_0} 
\]
\[ - c^{\frac1\beta}(2\overline{J}(\tfrac{\underline{x}+\underline{y}}2,\tfrac{\overline{x}+\overline{y}}2)-Q_\beta(\underline{x}))^{\tfrac1\beta-1} \overline{|J'|}(\tfrac{\underline{x}+\underline{y}}2,\tfrac{\overline{x}+\overline{y}}2)\mathbf{1}_{\text{not}\;\frac{\overline{x}+\overline{y}}{2}<x_0}  \]
\[ - c^{\frac1\beta} (2\overline{J}(\tfrac{\underline{x}+\underline{y}}2,\tfrac{\overline{x}+\overline{y}}2)-Q_\beta(\underline{x}))^{\tfrac1\beta-1}Q'_\beta(\underline{x}). \]
Running $\textsc{Partition}_2(\underline{g_{QJ,1,\beta,c}},[\frac14,\frac12]\times [\frac12,\frac58])$ shows  $(\beta,c)\in \{(\beta_0,1),(\frac12,c_0)\}$,
\begin{equation}\label{eqn:g_QJ_1_auto}\auto
-\beta\partial_x g_{QJ,\beta,c} > 10^{-5}
\end{equation}
on this region.
Thus it only remains to check that
$g_{QJ,\beta,c}(\tfrac12, y)\ge 0$
for all $y\in [\tfrac12,\tfrac58]$ and $(\beta,c)\in \{(\beta_0,1),(\frac12,c_0)\}$, but this is equivalent to showing $G^1_\beta[c J](\frac12,y)\ge 0$ for these values, which already follows from Proposition \ref{prop:caseJ}. 

\subsubsection{Far from diagonal: $\tfrac14\leq x\leq \tfrac12, \frac58\leq y\le 1$, $x+y\ge 1$}
\label{sec:QJ-II}
We will show the following claim that is stronger than required:
\begin{prop}
For all $x\in [\frac14, \frac12]$ and $y\in[\frac58,1]$,
\begin{equation}\label{eqn:g_QJ_2_auto}\auto
g_{QJ,2}(x,y)= ((y-x)^2 +J(y)^2)^{1/2}+Q_{\frac12}(x)-2J(\tfrac{x+y}{2})>10^{-7}.
\end{equation}
\end{prop}

\begin{proof}
A tight lower bound is given by
\[ \underline{g_{QJ,2}}(\underline{x},\overline{x},\underline{y},\overline{y}) = ((\underline{y}-\overline{x})^2 + J(\overline{y})^2)^{1/2} + Q_{\frac12}(\underline{x})-2\overline{J}(\tfrac{\underline{x}+\underline{y}}{2},\tfrac{\overline{x}+\overline{y}}{2}). \]
Running $\textsc{Partition}_2(\underline{g_{QJ,2}},[\tfrac14,\tfrac12]\times [\tfrac58,1])$ shows the claim.
\end{proof}

\section{Proof of the Poincar\'e inequality}\label{sec:poincare}

Every Boolean-valued function $f$ can be written as $f=\mathbf{1}_A$ for some $A\subset \{0,1\}^n$.
Then $\mathbf{E} f=|A|=2^{-n}\#A$ and
\[ \|f-\mathbf{E} f\|^p_p = |A|^{p} (1-|A|) + |A| (1-|A|)^{p}. \]
On the other side of the inequality, 
\begin{equation}\label{eqn:twosidedgradient}
\|\nabla f\|_p^p = 2^{-p} ( \mathbf{E} h_A^{p/2} + \mathbf{E} h_{A^c}^{p/2}).  
\end{equation}
Set $p=2\beta$.
If $\mathbf{E}h_A^{\beta}\ge B(|A|)$ for some function $B$, then
\[\|\nabla f\|_p^p \geq 2^{-2\beta}(B(|A|)+B(1-|A|)).\]
Thus Theorem \ref{thm:poincare} follows from \eqref{eqn:actualisoperimbeta0} and \eqref{eqn:actualisoperimhalf} if we show the following.

\begin{prop}\label{prop:poincarepenult}
For all $x\in [0,1]$, $\beta\in [\frac12,\beta_0]$:
\[ G_{P,\beta}(x)= 2^{-2\beta}( b_\beta(x) + b_\beta(1-x)) - x^{2\beta}(1-x) - x(1-x)^{2\beta} \ge 0. \]
\end{prop}
\begin{rem}
The conclusion holds for all $\beta\in[\tfrac12,1]$ and this can be proved by the same methods.
\end{rem}

\begin{proof}[Proof of Proposition \ref{prop:poincarepenult}]
Since $G_{P,\beta}(x)=G_{P,\beta}(1-x)$ it suffices to show this for $x\in [0,\tfrac12]$.
Notice that $G_{P,\beta}$ vanishes at $x=0,\tfrac12,1$.

\subsubsection*{Case I: $x\in [0,\tfrac1{64}]$}
We need to show
\[2^{-2\beta}( L_\beta(x) + J(1-x)) - x^{2\beta}(1-x) - x(1-x)^{2\beta} \ge 0\]
It is clear that this inequality holds asymptotically as $x\to 0^+$ and this is not sensitive to the value of $\beta$.
Using $\beta\in[\frac12,\beta_0]$ and Corollary \ref{prop:Jlowerbd}, the left hand side is
\begin{equation*}
\ge 2^{-2\beta_0}(L_{\frac12}(x) + J(1-x)) - 2x(1-x)
\end{equation*}
\[ \ge x (2^{-2\beta_0} \sqrt{\log_2(1/x)} + 2^{-2\beta_0} \sqrt{\log(w_0/x)}- 2) = x\cdot  g_{P,1}(x),  \]
Notice that $g_{P,1}$ is a decreasing function of $x$ and one can evaluate 
\begin{equation}
g_{P,1}(\tfrac1{64})>0.2
\end{equation}
Thus, $G_{P,\beta}(x)\ge 0.2\, x$ for all $x\in [0,\tfrac1{64}]$ and $\beta\in[\frac12,\beta_0]$.

\subsubsection*{Case II: $x\in [\tfrac1{64},\tfrac14]$}
Note that $x\mapsto J(1-x)$ is increasing on $x\in [\frac1{64},\frac14]$ because $x_0<1-\tfrac14$.
Thus a tight lower bound for $G_{P,\beta}(x)$ is given by
\[\underline{g_{P,2}}(\underline{x},\overline{x})= 2^{-2\beta_0}(L_{\frac12}(\underline{x}) + J(1-\underline{x})) - 2\overline{x}(1-\underline{x}).\]
Running $\textsc{Partition}_1(\underline{g_{P,2}},[\tfrac1{64},\tfrac14])$ shows that
\begin{equation}\label{eqn:g_P_2_auto}\auto
G_{P,\beta}(x)>10^{-4}
\end{equation}
for all $x\in [\frac1{64},\frac14]$ and $\beta\in[\frac12,\beta_0]$.

\subsubsection*{Case III: $x\in [\tfrac14,\tfrac12]$}
We need to show
\[2^{-2\beta}( Q_\beta(x) + J(1-x)) - x^{2\beta}(1-x) - x(1-x)^{2\beta} \ge 0.\]
Since the left-hand side equals $0$ at $x=\tfrac12$ it will suffice to show that the function on the left is decreasing in $x$, i.e. that
\[ g_{P,3,\beta}(x)=-\partial_x G_{P,\beta}(x)>0\]
for all $x\in [\frac14,\frac12]$, $\beta\in[\frac12,\beta_0]$.
Calculate
\[ g_{P,3,\beta}(x)=-2^{-2\beta} Q'_\beta(x)+2^{-2\beta} J'(1-x) + 2\beta x^{2\beta-1} (1-x) - x^{2\beta} \] 
\[ + (1-x)^{2\beta} - 2\beta x(1-x)^{2\beta-1}. \]
$J'(1-x)$ is increasing for  $x\in [0, \tfrac12]$, positive if $1-x<x_0$ and negative if $1-x>x_0$.
Also,
$Q'_\beta(x)$ is decreasing in $x$, positive and decreasing in $\beta$ by Lemma \ref{lem:Qbprop}.
A tight lower bound for $g_{P,3,\beta}$ is therefore given by
\[ \underline{g_{P,3}}(\underline{x},\overline{x}) = - 2^{-1} Q'_{\frac12}(\underline{x}) + 2^{-2\beta_0} J'(1-\underline{x}) \mathbf{1}_{1-\underline{x}<x_0} \]
\[- 2^{-1}\overline{|J'|}(1-\overline{x},1-\underline{x})\mathbf{1}_{\text{not}\;1-\underline{x}<x_0} \]
\[ + \underline{x}^{2\beta_0-1}(1-\overline{x}) - \overline{x} + (1-\overline{x})^{2\beta_0}-2\beta_0\,\overline{x}  \]
Running $\textsc{Partition}_1(\underline{g_{P,3}},[\frac14,\frac12])$ shows
\begin{equation}\label{eqn:g_P_3_auto}\auto
g_{P,3,\beta}(x) > 0.001,
\end{equation}
for all $x\in[\frac14,\frac12]$ which finishes the proof of Proposition \ref{prop:poincarepenult}.
\end{proof}

\end{document}